\documentclass[11pt,a4paper]{article}

\usepackage{amssymb,amsmath,amsthm}
\usepackage[all]{xy}

\theoremstyle{plain}
\newtheorem{thm}{Theorem}[section]
\newtheorem{prop}[thm]{Proposition}
\newtheorem{lem}[thm]{Lemma}
\newtheorem{cor}[thm]{Corollary}

\theoremstyle{definition}
\newtheorem{dfn}[thm]{Definition}
\newtheorem{rem}[thm]{Remark}

\numberwithin{equation}{section}

\makeatletter
\renewenvironment{proof}[1][\proofname]{\par
  \pushQED{\qed}%
  \normalfont \topsep6\p@\@plus6\p@\relax
  \trivlist
  \item[\hskip\labelsep
	\bfseries
    #1\@addpunct{.}]\ignorespaces
}{%
  \popQED\endtrivlist\@endpefalse
}
\makeatother

\DeclareMathOperator{\Hom}{Hom}
\DeclareMathOperator{\End}{End}
\DeclareMathOperator{\Ker}{Ker}
\DeclareMathOperator{\Ima}{Im}
\DeclareMathOperator{\ev}{ev}

\title{Affine Yangian action on the Fock space}

\author{Ryosuke Kodera}

\date{}

\makeatletter
\let\@old@@maketitle=\@maketitle
\def\@maketitle{%
\footnotetext{%
\hspace*{-1em}\hspace*{-\footnotesep}%
Research Institute for Mathematical Sciences, Kyoto University, Kyoto 606-8502, Japan.\\
E-mail address: kodera@kurims.kyoto-u.ac.jp
}
\@old@@maketitle
}
\makeatother

\begin{document}
\maketitle

\begin{abstract}
The localized equivariant homology of the quiver variety of type $A_{N-1}^{(1)}$ can be identified with the level one Fock space by assigning a normalized torus fixed point basis to certain symmetric functions, Jack($\mathfrak{gl}_N$) symmetric functions introduced by Uglov.
We show that this correspondence is compatible with actions of two algebras, the Yangian for $\mathfrak{sl}_N$ and the affine Lie algebra $\hat{\mathfrak{sl}}_N$, on both sides.
Consequently we obtain affine Yangian action on the Fock space.
\end{abstract}

\section{Introduction}

Yangian associated with an arbitrary semisimple Lie algebra was introduced by Drinfeld \cite{MR802128} as a one-parameter deformation of the universal enveloping algebra of the current Lie algebra.
Later Drinfeld \cite{MR914215} gave a new presentation for Yangians which is suitable to develop highest weight theory for their representations.

Drinfeld's new presentation is applicable to the data of the Cartan matrix of any symmetrizable Kac-Moody Lie algebra.
We then obtain Yangians associated with Cartan matrices of affine type as particular cases.
We call them affine Yangians.
By this definition affine Yangian may be regarded as a degeneration of quantum toroidal algebra.
In the case of affine type A, associated Yangian involves two parameters and the defining relations were first given by Guay \cite{MR2199856}.

Such a generalization of Yangian seems to be natural from the theory of quiver varieties.
Varagnolo \cite{MR1818101} constructed an action of the Yangian for a symmetric Kac-Moody Lie algebra on the sum of the equivariant homology groups of quiver varieties.
An original motivation of Nakajima \cite{MR1302318,MR1604167} to introduce quiver varieties comes from his study of moduli spaces of instantons on ALE spaces with Kronheimer \cite{MR1075769}.
In the original situation, the corresponding quivers are of affine type, therefore quantum toroidal algebras and affine Yangians have particular importance from this viewpoint.
The appearance of two parameters in the affine type A case can be explained as the quiver variety of affine type A admits an action of the two-dimensional torus.

Study of the representation theory of affine Yangians has not been done sufficiently.
Guay \cite{MR2199856} established a Schur-Weyl type duality between the affine Yangian and the trigonometric double affine Hecke algebra.
An action of the affine Yangian on the equivariant cohomology of the affine Laumon space was constructed by Feigin-Finkelberg-Negut-Rybnikov \cite{MR2827177}.

Our aim is to study the action constructed by Varagnolo more explicitly.
The localized equivariant homology of the quiver variety of type $A_{N-1}^{(1)}$ has a basis consisting of the classes of the torus fixed points parametrized by partitions.
We can derive an explicit formula for the action of the affine Yangian on the fixed point basis (Proposition~\ref{prop:formula} and Theorem~\ref{thm:formula}).
Our main result relates this affine Yangian module to the level one Fock space.

We turn to explain Yangian action on the Fock space due to Uglov \cite{MR1724950}.
Takemura-Uglov \cite{MR1457972} used a result of Nazarov-Tarasov \cite{MR1605817} to study Yangian symmetry of the spin Calogero-Sutherland model.
They described the decomposition of the space $F_n$ of states of the spin $N$ Calogero-Sutherland model with $n$ particles as a module over the Yangian for $\mathfrak{gl}_N$.
Modules appearing in the decomposition belong to a class studied in \cite{MR1605817} and in particular they have remarkable bases, called Gelfand-Zetlin bases.
The level one Fock space is realized as a certain limit $n \to \infty$ of $F_n$, and Uglov \cite{MR1724950} constructed an action of the Yangian for $\mathfrak{gl}_N$ on the Fock space.
The limit of the Gelfand-Zetlin bases form an orthogonal basis of the Fock space, which Uglov called Jack($\mathfrak{gl}_N$) symmetric functions since these functions reduce to the Jack symmetric functions when $N=1$.
We give an explicit formula for the actions of the Drinfeld generators of the Yangian for $\mathfrak{sl}_N$ on the basis of Jack($\mathfrak{gl}_N$) symmetric functions in Theorem~\ref{thm:main1}.

We can ask whether the actions of the Yangian for ${\mathfrak{sl}}_N$ and the affine Lie algebra $\hat{\mathfrak{sl}}_N$ on the Fock space can be glued and extended to an action of the affine Yangian of type $A_{N-1}^{(1)}$.
In fact, this has been settled for the nondegenerate case, that is, for the quantum toroidal algebra by Varagnolo-Vasserot \cite{MR1626481} and Saito-Takemura-Uglov \cite{MR1603798}, both based on a work of Takemura-Uglov \cite{MR1600311}.
Hence the quantum toroidal algebra has been known to act on the level one $q$-Fock space.
Moreover Nagao \cite{MR2583334} proved that this module is isomorphic to the equivariant $K$-group of the quiver variety, where an action of the quantum toroidal algebra on the latter is due to Varagnolo-Vasserot \cite{MR1722361}.

In this paper, we give an affirmative answer for the affine Yangian case by a different approach.
Since we know that the affine Yangian acts on the localized equivariant homology of the quiver variety, the Yangian for ${\mathfrak{sl}}_N$ and the affine Lie algebra $\hat{\mathfrak{sl}}_N$ act on the same space.
Hence it is enough to construct an isomorphism between the localized equivariant homology and the Fock space which is compatible with the actions of two algebras.
A candidate is found by norm formulas for the bases on both sides.
The norm of Jack($\mathfrak{gl}_N$) symmetric function has been known and it turns out to be exactly the same as that of the normalized basis element corresponding to the torus fixed point of the quiver variety of type $A_{N-1}^{(1)}$.
Our main result is Theorem~\ref{thm:main2}, where we show that the assignment of the normalized fixed point basis to the Jack($\mathfrak{gl}_N$) symmetric functions satisfies the desired properties.

\begin{thm}[Theorem~\ref{thm:main2}]
The assignment
\[
b_{\lambda} = (-1)^{\varepsilon_{\lambda}} \left( \displaystyle \prod_{\substack{s \in \lambda\\ h_{\lambda}(s) \equiv 0 \bmod N}} \dfrac{1}{\varepsilon_1 (l_{\lambda}(s) + 1) - \varepsilon_2 a_{\lambda}(s)} \right) [\lambda] \mapsto P_{\lambda}
\]
of the normalized torus fixed point basis to the Jack$(\mathfrak{gl}_N)$ symmetric functions gives
\begin{itemize}
\item a $\mathbb{C}(\varepsilon_1, \varepsilon_2)$-linear isometry,
\item an isomorphism of modules over the Yangian $Y_{\hbar}(\mathfrak{sl}_N)$,
\item an isomorphism of modules over the affine Lie algebra $\hat{\mathfrak{sl}}_N$
\end{itemize}
between the localized equivariant homology $H =\bigoplus_{{\bf v}} H_*^{T}(\mathfrak{M}({\bf v})) \otimes_{\mathbb{C}[\varepsilon_1, \varepsilon_2]} \mathbb{C}(\varepsilon_1, \varepsilon_2)$ of the quiver variety of type $A_{N-1}^{(1)}$ and the Fock space $F$.
\end{thm}

\begin{cor}[Corollary~\ref{cor:cor}]
The actions of the Yangian $Y_{\hbar}(\mathfrak{sl}_N)$ and the affine Lie algebra $\hat{\mathfrak{sl}}_N$ on the Fock space $F$ can be uniquely extended to an action of the affine Yangian $Y_{\varepsilon_1, \varepsilon_2}(\hat{\mathfrak{sl}}_N)$.
\end{cor}

The main result of this paper can be regarded as a degenerate analog of that of Nagao \cite{MR2583334} as explained.
The idea of comparing distinguished bases on both sides follows the paper \cite{MR2583334}, while an approach for actual calculation of the action on the basis of the Fock space is different.
We do not know whether our result can be derived directly from the quantum toroidal case by a certain degenerate procedure.

We note that a relation between the basis corresponding to the torus fixed points of the quiver variety of type $A_{N-1}^{(1)}$ and Uglov's Jack($\mathfrak{gl}_N$) symmetric functions has been already mentioned by Belavin-Bershtein-Tarnopolsky \cite{MR3046746}.

This paper is organized as follows.
In Section~\ref{sec:Uglov}, we review the Yangians associated with $\mathfrak{gl}_N$ and $\mathfrak{sl}_N$, and recall Uglov's construction of an action of the Yangian for $\mathfrak{gl}_N$ on the level one Fock space.
We also introduce Jack($\mathfrak{gl}_N$) symmetric functions, a distinguished basis of the Fock space.
In Section~\ref{sec:explicit}, we give an explicit formula for the actions of the Drinfeld generators of the Yangian for $\mathfrak{sl}_N$ on the Jack($\mathfrak{gl}_N$) symmetric functions using the theory of Gelfand-Zetlin basis developed by Nazarov-Tarasov.
In Section~\ref{sec:quiver}, we recall basic facts on the quiver variety of affine type A, including a description of the torus fixed points and the corresponding basis for the localized equivariant homology.
Then we make an observation that they can be identified with the basis of the Jack($\mathfrak{gl}_N$) symmetric functions for the Fock space after a normalization.
The affine Yangian and its action on the localized equivariant homology are recalled in Section~\ref{sec:Fock}.
We calculate the action explicitly in terms of the fixed point basis.
An argument on a normalization of signs in the formula is given.
Then comparing the actions of the Yangian for $\mathfrak{sl}_N$ and the affine Lie algebra $\hat{\mathfrak{sl}}_N$ on the Fock space and on the localized  equivariant homology, we conclude that the identification of them is compatible with those actions.
Consequently, we obtain affine Yangian action on the Fock space.

\subsection*{Acknowledgments}

The author is grateful to Hiraku Nakajima who taught him many things on quiver varieties.
His explanation of Jack symmetric functions was also helpful.
The author would like to thank Shintarou Yanagida for his interest and various discussions, which especially improved the understanding of the sign normalization of the fixed point basis.
He also thanks Yoshihisa Saito and Katsuyuki Naoi for discussions and comments.

The author is supported by Research Fellowships of the Japan Society for the Promotion of Science for Young Scientists.

\section{Uglov's construction of Yangian action}\label{sec:Uglov}

\subsection{Yangian for $\mathfrak{gl}_N$ and $\mathfrak{sl}_N$}

We introduce the Yangians associated with $\mathfrak{gl}_N$ and $\mathfrak{sl}_N$, and then recall a relation between them.
A standard reference is Molev's book \cite{MR2355506}.
We work over the field $\mathbb{C}(\varepsilon_1, \varepsilon_2)$ of rational functions in two parameters throughout the paper.

\begin{dfn}\label{dfn:gl}
Let $\hbar \in \mathbb{C}(\varepsilon_1, \varepsilon_2)$ be a parameter.
The Yangian $Y_{\hbar}(\mathfrak{gl}_N)$ is the algebra over $\mathbb{C}(\varepsilon_1, \varepsilon_2)$ generated by $T_{ij}^{(r)}$ $(1 \leq i,j \leq N, r \in \mathbb{Z}_{\geq 1})$ subject to the relations:
\[
[T_{ij}^{(r)}, T_{kl}^{(s)}] = \delta_{kj} T_{il}^{(r+s-1)} - \delta_{il} T_{kj}^{(r+s-1)} + \hbar \sum_{a=2}^{\min\{r,s\}} ( T_{kj}^{(a-1)} T_{il}^{(r+s-a)} - T_{kj}^{(r+s-a)} T_{il}^{(a-1)}).
\]
\end{dfn}

\begin{rem}
The notation $T_{ij}^{(r)}$ used by Nazarov-Tarasov~\cite{MR1605817} and Uglov~\cite{MR1724950} corresponds to our $T_{ji}^{(r)}$.
\end{rem}

Consider formal series
\[
T_{ij}(u) = \delta_{ij} + \hbar \sum_{r \geq 1} T_{ij}^{(r)} u^{-r}
\]
in $Y_{\hbar}(\mathfrak{gl}_N)[[u^{-1}]]$.
Then the above relation is equivalent to
\[
(u - v)[T_{ij}(u), T_{kl}(v)] = \hbar (T_{kj}(u) T_{il}(v) - T_{kj}(v) T_{il}(u)).
\]
A coproduct on $Y_{\hbar}(\mathfrak{gl}_N)$ is given by
\[
\Delta(T_{ij}(u)) = \sum_{k=1}^{N} T_{ik}(u) \otimes T_{kj}(u).
\]
We regard $U(\mathfrak{gl}_N)$ as a subalgebra of $Y_{\hbar}(\mathfrak{gl}_N)$ via
\[
E_{ij} \mapsto T_{ij}^{(1)}.
\]

There exists the algebra homomorphism $\ev_a$ for each $a \in \mathbb{C}(\varepsilon_1, \varepsilon_2)$ defined by
\begin{align*}
\ev_a \colon Y_{\hbar}(\mathfrak{gl}_N) &\to U(\mathfrak{gl}_N) \otimes \mathbb{C}(\varepsilon_1, \varepsilon_2) \\
T_{ij}^{(r)} &\mapsto a^{r-1} E_{ij}.
\end{align*}
Let $\varpi_1, \ldots, \varpi_N$ be the fundamental weights of $\mathfrak{gl}_N$ and $V(\varpi_i)$ the corresponding simple module.
We define the fundamental module $V(\varpi_i)_a$ of $Y_{\hbar}(\mathfrak{gl}_N)$ as the pullback of $V(\varpi_i)$ by $\ev_a$.
Let
\[
f(u) = 1 + \hbar \sum_{r \geq 1} f_r u^{-r} \in \mathbb{C}(\varepsilon_1, \varepsilon_2)[[u^{-1}]]
\]
be a formal power series in $u^{-1}$.
The map
\[
\omega_{f} \colon T_{ij}(u) \mapsto f(u) T_{ij}(u)
\]
defines an algebra automorphism of $Y_{\hbar}(\mathfrak{gl}_N)$.
Thus we can consider the pullback $\omega_{f}^* V$ for any $Y_{\hbar}(\mathfrak{gl}_N)$-module $V$.

\begin{dfn}\label{dfn:sl}
Let $\hbar \in \mathbb{C}(\varepsilon_1, \varepsilon_2)$ be a parameter.
The Yangian $Y_{\hbar}(\mathfrak{sl}_N)$ is the algebra over $\mathbb{C}(\varepsilon_1, \varepsilon_2)$ generated by $X_{i,r}^{\pm}, H_{i,r}$ $(1 \leq i \leq N-1, r \in \mathbb{Z}_{\geq 0})$ subject to the relations:
\[
[H_{i,r}, H_{j,s}] = 0,
\]
\[
[X_{i,r}^{+}, X_{j,s}^{-}] = \delta_{ij} H_{i, r+s},
\]
\[
[H_{i,0}, X_{j,s}^{\pm}] = \pm a_{ij}X_{j,s}^{\pm},
\]
\[
[H_{i,r+1}, X_{j, s}^{\pm}] - [H_{i, r}, X_{j, s+1}^{\pm}] = \pm \dfrac{1}{2}\hbar a_{ij} (H_{i, r}X_{j, s}^{\pm} + X_{j, s}^{\pm}H_{i, r}),
\]
\[
[X_{i,r+1}^{\pm}, X_{j, s}^{\pm}] - [X_{i, r}^{\pm}, X_{j, s+1}^{\pm}] = \pm \dfrac{1}{2}\hbar a_{ij} (X_{i, r}^{\pm}X_{j, s}^{\pm} + X_{j, s}^{\pm}X_{i, r}^{\pm}),
\]
\[
\sum_{w \in \mathfrak{S}_{1-a_{ij}}} [X_{i,r_{w(1)}}^{\pm}, [X_{i,r_{w(2)}}^{\pm}, \dots, [X_{i,r_{w(1 - a_{ij})}}^{\pm}, X_{j,s}^{\pm}]\dots]] = 0 \ (i \neq j), 
\]
where
\[
a_{ij} =
\begin{cases}
2  \text{ if } i=j, \\
-1 \text{ if } i=j \pm 1, \\
0  \text{ otherwise.}
\end{cases}
\]
\end{dfn}

The subalgebra of $Y_{\hbar}(\mathfrak{sl}_N)$ generated by $X_{i,0}^{\pm}, H_{i,0}$ $(1 \leq i \leq N-1)$ is isomorphic to $U(\mathfrak{sl}_N)$.
The generators $X_{i,0}^{\pm}, H_{i,0}$ coincide with the standard Chevalley generators $X_{i}^{\pm}, H_{i}$ of $\mathfrak{sl}_N$.

For sequences $(i_1, \ldots, i_m)$ and $(j_1, \ldots, j_m)$ in $\{1, \ldots, N\}$, define the quantum minor $T_{(i_1, \ldots, i_m),(j_1, \ldots, j_m)} (u)$ by
\begin{align*}
& T_{(i_1, \ldots, i_m),(j_1, \ldots, j_m)}(u)\\
&= \sum_{w \in \mathfrak{S}_m} (-1)^{l(w)} T_{i_{w(1)}, j_1}(u) T_{i_{w(2)}, j_2}(u-\hbar) \cdots T_{i_{w(m)}, j_m}(u - \hbar(m-1)).
\end{align*}
It is known that
\[
T_{(i_1, \ldots, i_m),(j_1, \ldots, j_m)}(u) = \sum_{w \in \mathfrak{S}_m} (-1)^{l(w)} T_{i_1, j_{w(1)}}(u - \hbar(m-1)) \cdots T_{i_m, j_{w(m)}}(u).
\]
We define formal series $A_i(u)$, $B_{i}(u)$ and $C_{i}(u)$ by
\begin{align*}
A_{i}(u) &= T_{(1, \ldots, i),(1, \ldots, i)}(u), \\
B_{i}(u) &= T_{(1, \ldots, i),(1, \ldots, i-1, i+1)}(u), \\
C_{i}(u) &= T_{(1, \ldots, i-1, i+1),(1, \ldots, i)}(u).
\end{align*}
The Gelfand-Zetlin subalgebra $A_{\hbar}(\mathfrak{gl}_N)$ of the Yangian $Y_{\hbar}(\mathfrak{gl}_N)$ is defined to be generated by all coefficients of $A_i(u)$ for all $i$.
It is known that $A_{\hbar}(\mathfrak{gl}_N)$ is commutative.

We introduce elements $X_{i,r}^{\pm}, H_{i,r}$ of $Y_{\hbar}(\mathfrak{gl}_N)$ by the following equations (see \cite[Remark~3.1.8]{MR2355506} ):
\begin{align*}
\hbar \sum_{r \geq 0} X_{i,r}^+ u^{-r-1} &= C_i(u + \dfrac{1}{2}\hbar(i-1)) A_i(u + \dfrac{1}{2}\hbar(i-1))^{-1}, \\
\hbar \sum_{r \geq 0} X_{i,r}^- u^{-r-1} &= A_i(u + \dfrac{1}{2}\hbar(i-1))^{-1} B_i(u + \dfrac{1}{2}\hbar(i-1)), \\
1 + \hbar \sum_{r \geq 0} H_{i,r} u^{-r-1} &= \dfrac{A_{i+1}(u + \dfrac{1}{2}\hbar(i+1)) A_{i-1}(u + \dfrac{1}{2}\hbar(i-1))}{A_i(u + \dfrac{1}{2}\hbar(i+1)) A_i(u + \dfrac{1}{2}\hbar(i-1))}.
\end{align*}

\begin{thm}[Drinfeld~\cite{MR914215}, see also \cite{MR2355506}]
The elements $X_{i,r}^{\pm}, H_{i,r}$ satisfy the defining relations of the Yangian $Y_{\hbar}(\mathfrak{sl}_N)$.
The subalgebra of $Y_{\hbar}(\mathfrak{gl}_N)$ generated by $X_{i,r}^{\pm}, H_{i,r}$ $(1 \leq i \leq N-1$, $r \in \mathbb{Z}_{\geq 0})$ is isomorphic to $Y_{\hbar}(\mathfrak{sl}_N)$.
\end{thm}

In the sequel, the parameter is taken to be $\hbar = \varepsilon_1 + \varepsilon_2$ or $\hbar' = - \hbar = - (\varepsilon_1 + \varepsilon_2)$.
We easily see that $Y_{\hbar}(\mathfrak{sl}_N)$ and $Y_{\hbar'}(\mathfrak{sl}_N)$ are isomorphic via
\begin{align*}
X_{i,r}^{\pm} &\mapsto X_{i,r}^{\mp}, \\
H_{i,r} &\mapsto -H_{i,r}.
\end{align*}
Hereafter we always use the symbols $X_{i,r}^{\pm}, H_{i,r}$ for the elements of $Y_{\hbar}(\mathfrak{sl}_N)$, but $T_{ij}^{(r)}$ and all related formal series for $Y_{\hbar'}(\mathfrak{gl}_N)$.
Then the composite of
\[
Y_{\hbar}(\mathfrak{sl}_N) \overset{\cong}{\to} Y_{\hbar'}(\mathfrak{sl}_N) \hookrightarrow Y_{\hbar'}(\mathfrak{gl}_N)
\]
is given by
\begin{align}
-\hbar \sum_{r \geq 0} X_{i,r}^+ u^{-r-1} &= A_i(u - \dfrac{1}{2}\hbar(i-1))^{-1} B_i(u - \dfrac{1}{2}\hbar(i-1)), \label{eq:X_{i,r}^+} \\
-\hbar \sum_{r \geq 0} X_{i,r}^- u^{-r-1} &= C_i(u - \dfrac{1}{2}\hbar(i-1)) A_i(u - \dfrac{1}{2}\hbar(i-1))^{-1}, \label{eq:X_{i,r}^-} \\
1 + \hbar \sum_{r \geq 0} H_{i,r} u^{-r-1} &= \dfrac{A_{i+1}(u - \dfrac{1}{2}\hbar(i+1)) A_{i-1}(u - \dfrac{1}{2}\hbar(i-1))}{A_i(u - \dfrac{1}{2}\hbar(i+1)) A_i(u - \dfrac{1}{2}\hbar(i-1))}. \label{eq:H_{i,r}}
\end{align}
We set
\[
H_i(u) = 1 + \hbar \sum_{r \geq 0} H_{i,r} u^{-r-1}.
\]
The inclusion $Y_{\hbar}(\mathfrak{sl}_N) \hookrightarrow Y_{\hbar'}(\mathfrak{gl}_N)$ is compatible with the standard inclusion $\mathfrak{sl}_N \hookrightarrow \mathfrak{gl}_N$, that is, the diagram
\[
\xymatrix{
Y_{\hbar}(\mathfrak{sl}_N) \ \ar@{^{(}->}[r] & \ Y_{\hbar'}(\mathfrak{gl}_N) \\
\mathfrak{sl}_N \ \ar@{^{(}->}[u] \ar@{^{(}->}[r] & \ \mathfrak{gl}_N \ar@{^{(}->}[u]
}
\]
commutes.
By the construction of the generators, the automorphism $\omega_{f}$ acts identically on $Y_{\hbar}(\mathfrak{sl}_N)$.

\subsection{Preliminaries on partitions}\label{subsection:preliminaries}

Let $\mathcal{P}$ be the set of partitions.
Each element of $\mathcal{P}$ is a nonincreasing sequence $\lambda = (\lambda_1 \geq \lambda_2 \geq \cdots)$ where $\lambda_a \in \mathbb{Z}_{\geq 0}$ and $\lambda_a = 0$ for all but finitely many $a$.
The maximum $a$ such that $\lambda_a \neq 0$ is called the length of $\lambda$ and denoted by $l(\lambda)$.
We denote by ${}^t \lambda$ the transpose of $\lambda$.
The dominance order on $\mathcal{P}$ is denoted by $>$.

We identify a partition $\lambda = (\lambda_a) \in \mathcal{P}$ with the subset
\[
\{ (x,y) \mid x = 0, \ldots, l(\lambda) - 1, y = 0, \ldots, \lambda_{x+1} - 1 \}
\]
of $(\mathbb{Z}_{\geq 0})^2$.
Although a partition $\lambda$ is usually identified with the Young diagram $\{ (x,y) \mid x = 1, \ldots, l(\lambda), y = 1, \ldots, \lambda_{x} \}$, we use the above identification in this paper.
This is suitable to write down various formulas arising from quiver variety.
Each $(x,y) \in (\mathbb{Z}_{\geq 0})^2$ is called a cell.
For a cell $s = (x,y) \in \lambda$, we define the arm length $a_{\lambda}$, the leg length $l_{\lambda}$, and the hook length $h_{\lambda}$ by
\begin{align*}
a_{\lambda}(s) &= \lambda_{x+1} - (y + 1), \\
l_{\lambda}(s) &= {}^t \lambda_{y+1} - (x + 1), \\
h_{\lambda}(s) &= a_{\lambda}(s) + l_{\lambda}(s) + 1.
\end{align*}
For a fixed $N$, the residue of a cell $(x,y) \in (\mathbb{Z}_{\geq 0})^2$ is
\[
y - x \mod N.
\]
An $i$-cell is a cell whose residue is $i$.

For a partition $\lambda \in \mathcal{P}$, we say that a cell $(x,y) \in(\mathbb{Z}_{\geq 0})^2$ is removable if $(x,y) \in \lambda$ and $(x+1, y), (x, y+1) \notin \lambda$.
We say that a cell $(x,y) \in (\mathbb{Z}_{\geq 0})^2$ is addable if either $(x,y) \notin \lambda$ and $(x-1, y), (x, y-1) \in \lambda$, or $(x,y) = (0,\lambda_1), (l(\lambda),0)$.
The sets of removable $i$-cells and addable $i$-cells for $\lambda$ are denoted by $R_{\lambda,i}$ and $A_{\lambda,i}$.
For $(x,y) \in R_{\lambda,i}$ define
\begin{align*}
R_{\lambda,i,(x,y)}^{l} &= \{ (x',y') \in R_{\lambda,i} \mid x' < x \}, \\
R_{\lambda,i,(x,y)}^{r} &= \{ (x',y') \in R_{\lambda,i} \mid x < x' \},
\end{align*}
and similarly for $(x,y) \in A_{\lambda,i}$ define
\begin{align*}
A_{\lambda,i,(x,y)}^{l} &= \{ (x',y') \in A_{\lambda,i} \mid x' < x \}, \\
A_{\lambda,i,(x,y)}^{r} &= \{ (x',y') \in A_{\lambda,i} \mid x < x' \}.
\end{align*}

Given a nondecreasing sequence of integers ${\bf m} = (m_1 \leq m_2 \leq \cdots)$, we can write it as ${\bf m} = ((r_1)^{p_1}, (r_2)^{p_2}, \ldots)$ with $r_a < r_{a+1}$.
Define the sequence ${\bf m}^0$ as ${\bf m}^0 = (1^N 2^N \dots)$.
We denote by $M$ the set of nondecreasing sequence of integers satisfying
\begin{itemize}
\item $p_a \leq N$ for all $a$,
\item $m_a = m_a^0$ for all but finitely many $a$.
\end{itemize}
This set is introduced in \cite[Section~10, Definition~5]{MR1724950} and is denoted by $W$ there.
For each element ${\bf m} \in M$, we take the smallest $l$ such that
\[
{\bf m} = ((r_1)^{p_1},(r_2)^{p_2}, \ldots, (r_l)^{p_l}, (r_{l+1})^N, \ldots)
\]
and $m_a = m_a^0$ for $a > p_1 + \cdots + p_l$.
It is clear that $p_1 + \cdots + p_l$ is a multiple of $N$.
For $\lambda \in \mathcal{P}$, define $j(\lambda)_a \in \{1, \ldots, N\}$ and $m(\lambda)_a \in \mathbb{Z}$ so that
\[
\lambda_{a} - a + 1 = j(\lambda)_a - N m(\lambda)_a.
\]
Then the sequence $(m(\lambda)_a)$ is an element of $M$.
Note that $j(\lambda)_a > j(\lambda)_{a+1}$ if $m(\lambda)_a = m(\lambda)_{a+1}$.

We have a bijection between the sets $R_{\lambda,i}$ and
\begin{equation}
\left\{a \, \middle|\, 1 \leq a \leq l(\lambda),\ j(\lambda)_a = i+1, \text{ and either } 
\begin{array}{l}
\bullet\, m(\lambda)_a < m(\lambda)_{a+1}; \text{ or }\\
\bullet\, m(\lambda)_a = m(\lambda)_{a+1} \text{ and } j(\lambda)_{a+1} < i
\end{array}
\right\}\label{eq:SetR}
\end{equation}
via $(x,y) = (a-1, \lambda_a - 1)$.
Indeed any removable cell is of the form $(a-1, \lambda_a - 1)$ for some $a$ such that $\lambda_a > \lambda_{a+1}$.
The condition $\lambda_a > \lambda_{a+1}$ is equivalent to that
\[
(j(\lambda)_a - 1) - j(\lambda)_{a+1} > N(m(\lambda)_a - m(\lambda)_{a+1}).
\]
Since the residue of $(a-1, \lambda_a - 1)$ is $i$ if and only if $j(\lambda)_{a} = i+1$, we see the above correspondence is one-to-one.
Similarly we have a bijection between the sets $A_{\lambda,i}$ and
\begin{equation}
\left\{a \, \middle|\, 1 \leq a \leq l(\lambda) + 1,\ j(\lambda)_a = i, \text{ and either } 
\begin{array}{l}
\bullet\, m(\lambda)_{a-1} < m(\lambda)_{a}; \text{ or }\\
\bullet\, m(\lambda)_{a-1} = m(\lambda)_{a} \text{ and } j(\lambda)_{a-1} > i+1
\end{array}
\right\}\label{eq:SetA}
\end{equation}
via $(x,y) = (a-1, \lambda_a)$.
Indeed any addable cell is of the form $(a-1, \lambda_a)$ for some $a$ such that $\lambda_{a-1} > \lambda_{a}$.
The condition $\lambda_{a-1} > \lambda_{a}$ is equivalent to that
\[
j(\lambda)_{a-1} - (j(\lambda)_{a} + 1) > N(m(\lambda)_{a-1} - m(\lambda)_{a}).
\]
Since the residue of $(a-1, \lambda_a)$ is $i$ if and only if $j(\lambda)_{a} = i$, we see the above correspondence is one-to-one.
We note that
\begin{equation}
y - x = i - N m(\lambda)_a \label{eq:(x,y)}
\end{equation}
holds in both cases.

We have a decomposition $\mathcal{P} = \bigsqcup_{{\bf m} \in M} \mathcal{P}_{\bf m}$, where $\mathcal{P}_{\bf m}$ is the subset of $\mathcal{P}$ which consists of partitions satisfying $(m(\lambda)_a) = {\bf m}$.

\subsection{Yangian action on the Fock space}

Let $F$ be the space of symmetric functions in variables $x_1, x_2, \ldots$ over the field $\mathbb{C}(\varepsilon_1, \varepsilon_2)$.
It has a basis $\{s_{\lambda}\}_{\lambda \in \mathcal{P}}$ consisting of the Schur symmetric functions and admits the action of the affine Lie algebra $\hat{\mathfrak{sl}}_N$ defined by
\begin{align}
X_i^{+} s_{\lambda} &= \sum_{s \in R_{\lambda,i}} s_{\lambda \setminus s}, \label{eq:affine1} \\
X_i^{-} s_{\lambda} &= \sum_{s \in A_{\lambda,i}} s_{\lambda \cup s}, \label{eq:affine2} \\
H_i s_{\lambda} &= (\# A_{\lambda,i} - \# R_{\lambda,i}) s_{\lambda}, \label{eq:affine3}
\end{align}
where $X_i^{\pm}, H_i$ $(i \in \mathbb{Z} / N \mathbb{Z})$ are the standard Chevalley generators of $\hat{\mathfrak{sl}}_N$.
We call $F = \bigoplus_{\lambda \in \mathcal{P}} \mathbb{C}(\varepsilon_1, \varepsilon_2) s_{\lambda}$ the level one Fock space.

We recall a construction of an action of the Yangian $Y_{\hbar'}(\mathfrak{gl}_N)$ on the Fock space $F$ due to Uglov \cite{MR1724950}.
Fix $n \in \mathbb{Z}_{\geq 1}$.
We define the Dunkl-Cherednik operators $D_i$ $(1 \leq i \leq n)$ acting on $\mathbb{C}(\varepsilon_1, \varepsilon_2) [z_1^{\pm 1}, \ldots, z_n^{\pm 1}]$ by
\[
D_i = t z_i\dfrac{\partial}{\partial z_i} + c\left( \sum_{j < i} \dfrac{z_i}{z_i - z_j} (1 - K_{ji}) + \sum_{i<j} \dfrac{z_j}{z_i - z_j}(1 - K_{ij}) + n - i - \dfrac{1}{2} \right)
\]
with parameters $t,c \in \mathbb{C}(\varepsilon_1, \varepsilon_2)$.
Here $K_{ij}$ denotes the permutation of $z_i$ and $z_j$.
This is shifted by the constant $c((1/2)n - 1)$ from the usual Dunkl-Cherednik operator (see \cite[Part I, Definition~2.4]{MR1805058})
\[
t z_i\dfrac{\partial}{\partial z_i} + c\left( \sum_{j < i} \dfrac{z_i}{z_i - z_j} (1 - K_{ji}) + \sum_{i<j} \dfrac{z_j}{z_i - z_j}(1 - K_{ij}) + \dfrac{1}{2}(n - 2i + 1) \right),
\]
where $R^+$ in \cite{MR1805058} is taken to be the set of standard negative roots $\{ e_i - e_j \mid i > j \}$ of $\mathfrak{gl}_N$.
Uglov~\cite[(4.6)]{MR1724950} used
\[
\beta^{-1} z_i\dfrac{\partial}{\partial z_i} + \left( \sum_{j < i} \dfrac{z_i}{z_i - z_j} (1 - K_{ji}) + \sum_{i<j} \dfrac{z_j}{z_i - z_j}(1 - K_{ij}) + n - i \right).
\]
Hence we identify the parameters by $\beta = c/t$.

We have the relations
\[
K_{i,i+1} D_i - D_{i+1} K_{i,i+1} = c
\]
for $i = 1, \ldots, n-1$.
Thus we can define a right action of the degenerate affine Hecke algebra, generated by $s_i$ $(1 \leq i \leq n-1),$ $u_j$ $(1 \leq j \leq n)$ subject to the relations
\[
s_i^2 = 1,\ s_i s_j = s_j s_i \ (i \neq j, j \pm 1),\ s_i s_{i+1} s_i = s_{i+1} s_i s_{i+1},
\]
\[
u_i u_j = u_j u_i,\ s_i u_j = u_j s_i \ (i \neq j, j \pm 1),\ s_i u_i - u_{i+1} s_i = c,
\]
on $\mathbb{C}(\varepsilon_1, \varepsilon_2) [z_1^{\pm 1}, \ldots, z_n^{\pm 1}]$ by
\begin{align*}
s_i &\mapsto - K_{i,i+1}, \\
u_i &\mapsto - D_i.
\end{align*}
Set $c = \hbar' = -\hbar$ to apply Drinfeld's result.
Let $V$ be the $N$-dimensional vector representation of $\mathfrak{gl}_N$ with a standard basis $v_1, \ldots, v_N$.
By the Drinfeld correspondence~\cite{MR831053} (see also \cite{MR1706920}), we obtain a one-parameter family of actions of $Y_{\hbar'}(\mathfrak{gl}_N)$ on
\[
F_n = \mathbb{C}(\varepsilon_1, \varepsilon_2) [z_1^{\pm 1}, \ldots, z_n^{\pm 1}] \otimes _{\mathbb{C}\mathfrak{S}_n} (V^{\otimes n}),
\]
which depend on $t$.

We denote the tensor product $V \otimes \mathbb{C}(\varepsilon_1, \varepsilon_2) [z^{\pm 1}]$ by $V[z^{\pm 1}]$ and its element $v_j \otimes z^m$ by $v_j z^m$.
We have an obvious identification
\[
F_n \overset{\cong}{\to} \bigwedge^n V[z^{\pm 1}],
\]
\[ 
z_1^{m_1} \cdots z_n^{m_n} \otimes ( v_{j_1} \otimes \cdots \otimes v_{j_n} )  \mapsto v_{j_1} z^{m_1} \wedge \cdots \wedge v_{j_n} z^{m_n}.
\]
The latter wedge space has a basis $\{ v_{j_1} z^{m_1} \wedge \cdots \wedge v_{j_n} z^{m_n} \mid j_1 - N m_1 > j_2 - N m_2 > \cdots > j_n - N m_n \}$.
It is isomorphic to the space of symmetric Laurent polynomials:
\[
\bigwedge^n V[z^{\pm 1}] \overset{\cong}{\to} \mathbb{C}(\varepsilon_1, \varepsilon_2)[x_1^{\pm 1}, \ldots, x_n^{\pm 1} ]^{\mathfrak{S}_n},
\]
\begin{equation}
v_{j_1} z^{m_1} \wedge \cdots \wedge v_{j_n} z^{m_n} \mapsto s_{\lambda} \label{eq:wedge-Schur}
\end{equation}
by assigning the basis element $v_{j_1} z^{m_1} \wedge \cdots \wedge v_{j_n} z^{m_n}$ with $j_1 - N m_1 > j_2 - N m_2 > \cdots > j_n - N m_n$ to the Schur symmetric Laurent polynomial $s_{\lambda}$ associated with $\lambda = (\lambda_1 \geq \cdots \geq \lambda_n) \in \mathbb{Z}^n$ which is given by
\[
\lambda_a - a +1 = j_a - N m_a.
\]
See \cite[(8.7) and (8.9)]{MR1724950}.
Then we can take the subspace $F_n^{0}$ of $F_n$ which is isomorphic to the space of symmetric polynomials~\cite[(8.41), (8.42)]{MR1724950}.
This $F_n^{0}$ has two gradings, which are called principal and homogeneous.
We denote by $F_n^{0, \langle d \rangle}$ the principal grading and by $F_n^{0, (d)}$ the homogeneous grading.
The principal grading corresponds to the obvious one on the symmetric polynomials.
Hence the Fock space $F$ is realized by taking the inverse limit of $F_n^0$ degreewise, that is,
\[
F = \bigoplus_{d \geq 0} \varprojlim_{n} F_n^{0, \langle d \rangle}.
\]
To define a Yangian action, we consider the homogeneous grading.
For each $d \geq 0$ and $n=rN$, the Yangian action on $F_{rN}$ preserves the subspace $F_{rN}^{0, (d)}$.
Moreover $(F_{rN}^{0, (d)})_{r}$ forms a inverse system.
Take an element $v = (v_r)_r$ of the inverse limit
\[
\displaystyle\varprojlim_{r} F_{rN}^{0, (d)}
\]
and define an action of $T_{ij}(u)$ by
\begin{equation}
T_{ij}(u) v = ( f(u; r) T_{ij}(u) v_r )_r, \label{eq:intertwining}
\end{equation}
where
\[
f(u; r) = \prod_{s=1}^r \dfrac{u + (t + \hbar' N)s - \hbar'}{u + (t + \hbar' N)s}
\]
as in \cite[(10.7)]{MR1724950}.
It is nontrivial and is proved by Uglov that this is compatible with the inverse system and extends to a well-defined action of $Y_{\hbar'}(\mathfrak{gl}_N)$ on the inverse limit.
See \cite[Section~10]{MR1724950} for details.
Finally the two gradings are compatible and the Fock space is also realized as
\[
F = \bigoplus_{d \geq 0} \varprojlim_{r} F_{rN}^{0, (d)}.
\]
Thus we have a one-parameter family of Yangian actions on $F$ which depend on $t$.
We set $t = N \varepsilon_2$ from now on.

Let us describe the action of $\mathfrak{gl}_N$ on $F_n$ obtained by restricting the Yangian action.
By \cite[(10.5)]{MR1724950}, the element $E_{ij}$ of $\mathfrak{gl}_N$ acts on $F_n \cong \bigwedge^n V[z^{\pm 1}]$ by
\[
E_{ij} ( v_{j_1} z^{m_1} \wedge \cdots \wedge v_{j_n} z^{m_n} ) = \sum_{a=1}^n v_{j_1} z^{m_1} \wedge \cdots \wedge E_{ij} ( v_{j_a} z^{m_a} ) \wedge \cdots \wedge v_{j_n} z^{m_n}.
\]
Note that Uglov's $T_{ij}^{(r)}$ in \cite[(10.5)]{MR1724950} corresponds to our $T_{ji}^{(r)}$.
Considering the identification given by (\ref{eq:wedge-Schur}), the above formula implies that the action of $\mathfrak{sl}_N$ on $F$ obtained by restricting the Yangian action coincides with the standard action given by (\ref{eq:affine1})--(\ref{eq:affine3}) for $i=1, \ldots, N-1$.

\begin{rem}
The Yangian module $F_n$ can be regarded as the space of states of the spin Calogero-Sutherland model with $n$ particles.
See \cite{MR1457972}, \cite{MR1608527}, \cite{MR1724950}.
\end{rem}

\subsection{Jack($\mathfrak{gl}_N$) symmetric functions}

For a given ${\bf m} \in M$, we take the smallest $l$ such that
\[
{\bf m} = ((r_1)^{p_1},(r_2)^{p_2}, \ldots, (r_l)^{p_l}, (r_{l+1})^N, \ldots)
\]
and $m_a = m_a^0$ for $a > p_1 + \cdots + p_l$ as in Subsection~\ref{subsection:preliminaries}.
Let $r$ be the integer such that $p_1 + \cdots + p_l = rN$.
For $s = 1, \ldots, l$, put
\[
a_s = t r_s + \hbar'(p_1 + \cdots + p_s - \dfrac{3}{2})
\]
as in \cite[Proposition~10.5]{MR1724950} with shift $-(1/2)\hbar'$ because of our choice of the constant part of the Dunkl-Cherednik operator.
By \cite[Proposition 10.5]{MR1724950}, the Fock space $F$ decomposes as
\[
F = \bigoplus_{{\bf m} \in M} F_{\bf m}
\]
and each $F_{\bf m}$ is isomorphic to
\[
\omega_{f(u; r)}^* ( V(\varpi_{p_1})_{a_1} \otimes \cdots \otimes V(\varpi_{p_l})_{a_l} )
\]
as a $Y_{\hbar'}(\mathfrak{gl}_N)$-module.
The twist by the automorphism $\omega_{f(u; r)}$ is caused by the definition of the action of $T_{ij}(u)$ in (\ref{eq:intertwining}).
See \cite[Proposition 7.7]{MR1724950} for a decomposition of the Yangian module $F_n$.

These modules appearing in the decomposition of $F$ belong to a class of simple $Y_{\hbar'}(\mathfrak{gl}_N)$-modules studied by Nazarov-Tarasov~\cite{MR1605817}.
In particular, all simultaneous eigenspaces of the Gelfand-Zetlin subalgebra $A_{\hbar'}(\mathfrak{gl}_N)$ are one-dimensional. 

Uglov introduced a unique basis $\{P_{\lambda}\}_{\lambda \in \mathcal{P}}$ of $F$ satisfying the following conditions:
\begin{enumerate}
\item[(U1)] $\{P_{\lambda}\}_{\lambda \in \mathcal{P}_{\bf m}}$ forms a basis of each $F_{\bf m}$,
\item[(U2)] $P_{\lambda} \in s_{\lambda} + \sum_{\mu < \lambda} \mathbb{C}(\varepsilon_1, \varepsilon_2) s_{\mu},$
\item[(U3)] each $P_{\lambda}$ is a simultaneous eigenvector of $A_{\hbar'}(\mathfrak{gl}_N)$.
\end{enumerate}
Let $p_{r}$ be the $r$-th power sum symmetric function and define $p_{\lambda}$ by
\[
p_{\lambda} = p_{\lambda_1} p_{\lambda_2} \cdots
\]
for each $\lambda \in \mathcal{P}$.
Then $\{ p_{\lambda} \}_{\lambda \in \mathcal{P}}$ forms a basis of $F$.
Uglov defined a symmetric bilinear form $\langle\ ,\  \rangle_{F}$ on $F$ by
\[
\langle p_{\lambda}, p_{\mu} \rangle_{F} = \delta_{\lambda \mu} z_{\lambda} (- \varepsilon_2 / \varepsilon_1)^{l_N(\lambda)},
\]
where $z_{\lambda} = \displaystyle\prod_i i^{m_i} m_i !$ for $\lambda = (1^{m_1} 2^{m_2} \ldots)$ and $l_N(\lambda) = \# \{a \mid \lambda_a \neq 0,\ \lambda_a \equiv 0 \bmod N\}$.

\begin{rem}
Let us explain how to identify the parameters in \cite{MR1724950} with ours.
Uglov's bilinear form~\cite[9.4]{MR1724950} is defined as
\[
\langle p_{\lambda}, p_{\mu} \rangle_{F} = \delta_{\lambda \mu} z_{\lambda} \gamma^{-l_N(\lambda)}
\]
where $\gamma = N \beta + 1$ is a parameter.
The parameter $\beta$ can be identified with our $c/t$.
Therefore we set $t = N \varepsilon_2$ and obtain
\begin{align*}
\gamma &= N \beta + 1 \\
&= - \hbar / \varepsilon_2 + 1 \\
&= - \varepsilon_1 / \varepsilon_2. 
\end{align*}
\end{rem}

Let us explain the relation between the symmetric function $P_{\lambda}$ and the Macdonald symmetric function following \cite[9.4]{MR1724950}.
The bilinear form $\langle \ ,\ \rangle_F$ is a certain limit of the Macdonald scalar product $\langle \ ,\ \rangle_{q,t}$ defined by
\[
\langle p_{\lambda}, p_{\mu} \rangle_{q,t} = \delta_{\lambda \mu} z_{\lambda} \prod_{a=1}^{l(\lambda)} \dfrac{1 - q^{\lambda_a}}{1 - t^{\lambda_a}}.
\]
To obtain $\langle \ ,\ \rangle_F$, we formally take the limit $p \to 1$ after putting $q = p\, \omega_N$, $t = p^{\gamma} \omega_N$ where $\omega_N$ denotes a primitive $N$-th root of unity and $\gamma = - \varepsilon_1 / \varepsilon_2$ as above.

\begin{prop}\label{prop:adjoint}
We have
\[
\langle T_{ij}^{(r)} v, w \rangle_F = \langle v, T_{ji}^{(r)} w \rangle_F.
\]
In particular, elements of the Gelfand-Zetlin subalgebra $A_{\hbar'}(\mathfrak{gl}_N)$ are self-adjoint with respect to the bilinear form $\langle\ ,\  \rangle_{F}$ and we have
\[
\langle X_{i,r}^+ v, w \rangle_F = \langle v, X_{i,r}^- w \rangle_F.
\]
\end{prop}

\begin{proof}
It is shown in \cite[Proposition~9]{MR1457972} that a certain scalar product $\langle \ ,\ \rangle_{F_n}$ (see \cite[2.1]{MR1457972} or \cite[2.1]{MR1724950} for the definition) defined on $F_n$ has a property
\[
\langle T_{ij}^{(r)} v, w \rangle_{F_n} = \langle v, T_{ji}^{(r)} w \rangle_{F_n}.
\]
By \cite[Proposition~8.1 and Subsection~8.4]{MR1724950}, this $\langle \ ,\ \rangle_{F_n}$ is identified with the limit, same as mentioned before this proposition, of another scalar product which is denoted by $\langle \ ,\ \rangle_n'$ in Macdonald's book \cite[Section~9]{MR1354144}.
Since the Macdonald scalar product on $F$ is the limit $n \to \infty$ of $\langle\ ,\ \rangle_n' / \langle 1, 1 \rangle_n'$ \cite[(9.9)]{MR1354144}, our bilinear form $\langle \ ,\ \rangle_{F}$ is the limit $n \to \infty$ of $\langle \ ,\ \rangle_{F_n} / \langle 1,1 \rangle_{F_n}$.
This implies the desired property of $\langle \ ,\ \rangle_{F}$.
\end{proof}

Since $\{P_{\lambda}\}$ are simultaneous eigenvectors of $A_{\hbar'}(\mathfrak{gl}_N)$ with one-dimensional eigenspaces and $A_{\hbar'}(\mathfrak{gl}_N)$ are self-adjoint with respect to $\langle\ ,\ \rangle_F$, we have the orthogonality of $\{P_{\lambda}\}$:
\begin{enumerate}
\item[(U4)] $\langle P_{\lambda},  P_{\mu} \rangle_F = 0$ if $\lambda \neq \mu$.\end{enumerate}
The conditions (U2) and (U4) uniquely characterize the basis $\{P_{\lambda}\}_{\lambda \in \mathcal{P}}$.
Note that this characterization coincides with the one for Jack symmetric functions when $N=1$.
By this reason, Uglov calls $P_{\lambda}$ the Jack($\mathfrak{gl}_N$) symmetric function.
A polynomial version of $P_{\lambda}$ was first defined in \cite{MR1608527}.
Specializing the parameter $\hbar = \varepsilon_1 + \varepsilon_2$ to $0$, the bilinear form $\langle \ ,\  \rangle_{F}$ becomes the usual bilinear form on $F$ which the Schur symmetric functions are orthonormal with respect to.
Hence the Jack($\mathfrak{gl}_N$) symmetric function $P_{\lambda}$ specializes to the Schur symmetric function $s_{\lambda}$ at $\hbar = 0$.

The orthogonality condition (U4) can be deduced also from the fact that $P_{\lambda}$ is obtained from the Macdonald symmetric function $P_{\lambda}(q,t)$ by taking the limit
\[
P_{\lambda} = \lim_{p \to 1} P_{\lambda}(p\, \omega_N, p^{\gamma} \omega_N).
\]
In particular, $P_{\lambda}$ can be defined without Yangian action.
We obtain a norm formula for Jack($\mathfrak{gl}_N$) symmetric functions
\begin{equation}
\langle P_{\lambda}, P_{\lambda} \rangle_F = \displaystyle \prod_{\substack{s \in \lambda\\ h_{\lambda}(s) \equiv 0 \bmod N}} \dfrac{\varepsilon_1 l_{\lambda}(s) - \varepsilon_2 (a_{\lambda}(s) + 1)}{\varepsilon_1 (l_{\lambda}(s) + 1) - \varepsilon_2 a_{\lambda}(s)} \label{eq:norm}
\end{equation}
by this description \cite[(9.46)]{MR1724950}.

The eigenvalue of the action of $A_i(u)$ on $P_{\lambda}$ has been calculated in \cite{MR1724950}.
Recall that for each $\lambda \in \mathcal{P}$, we associate $j(\lambda)_a \in \{1, \ldots, N\}$ and $m(\lambda)_a \in \mathbb{Z}$ so that
\[
\lambda_{a} - a + 1 = j(\lambda)_a - N m(\lambda)_a.
\]
The sequence $(m(\lambda)_a) \in M$ gives rise to the numbers $p_1, \ldots, p_l$ such that $p_1 + \cdots + p_l = rN$ for some $r$.
Then we have
\begin{equation}
A_i(u) P_{\lambda} = g_i(u) \displaystyle \prod_{a=1}^{rN} \dfrac{u + t m(\lambda)_a + \hbar'(a - \dfrac{3}{2} + \delta(j(\lambda)_a \leq i))}{u + t m(\lambda)_a+ \hbar'(a - \dfrac{3}{2})} P_{\lambda} \label{eq:eigenvalue}
\end{equation}
by \cite[Proposition 10.3]{MR1724950}.
Here
\[
g_i(u) = f(u;r)f(u-\hbar';r) \cdots f(u-\hbar'(i-1);r).
\]
The above formula is shifted by $-(1/2)\hbar'$ from \cite{MR1724950} because of the choice of the Dunkl-Cherednik operator.
We will derive an explicit formula for the actions of the generators $X_{i,r}^{\pm}, H_{i,r}$ on the basis $P_{\lambda}$ in the next section.

We end this subsection with some formulas.

\begin{lem}\label{lem:ratio}
\begin{enumerate}
Let $(x,y) \in R_{\lambda,i}$ and $\mu = \lambda \setminus (x,y)$.
Then we have
\item
\begin{align*}
& \dfrac{\displaystyle\prod_{\substack{s \in \lambda\\ h_{\lambda}(s) \equiv 0 \bmod N}} \varepsilon_1 l_{\lambda}(s) - \varepsilon_2 (a_{\lambda}(s) + 1)}{\displaystyle\prod_{\substack{s \in \mu\\ h_{\mu}(s) \equiv 0 \bmod N}} \varepsilon_1 l_{\mu}(s) - \varepsilon_2 (a_{\mu}(s) + 1)} \\
& \\
&= (-1)^{\# A_{\lambda, i, (x,y)}^{r} - \# R_{\mu, i, (x,y)}^{r}} \\
& \dfrac{\displaystyle\prod_{(x', y') \in A_{\lambda, i, (x,y)}^{l}} \varepsilon_1 (x - x') + \varepsilon_2 (y - y')}{\displaystyle\prod_{(x', y') \in R_{\mu, i, (x,y)}^{l}} \varepsilon_1 (x - x' - 1) + \varepsilon_2 (y - y' - 1)} \dfrac{\displaystyle\prod_{(x', y') \in A_{\lambda, i, (x,y)}^{r}} \varepsilon_1 (x - x' + 1) + \varepsilon_2 (y - y' + 1)}{\displaystyle\prod_{(x', y') \in R_{\mu, i, (x,y)}^{r}} \varepsilon_1 (x - x') + \varepsilon_2 (y - y')},
\end{align*}
\item
\begin{align*}
& \dfrac{\displaystyle\prod_{\substack{s \in \mu\\ h_{\mu}(s) \equiv 0 \bmod N}} \varepsilon_1 (l_{\mu}(s) + 1) - \varepsilon_2 a_{\mu}(s)}{\displaystyle\prod_{\substack{s \in \lambda\\ h_{\lambda}(s) \equiv 0 \bmod N}} \varepsilon_1 (l_{\lambda}(s) + 1) - \varepsilon_2 a_{\lambda}(s)} \\
& \\
&= (-1)^{\# A_{\lambda, i, (x,y)}^{r} - \# R_{\mu, i, (x,y)}^{r}} \\
& \dfrac{\displaystyle\prod_{(x', y') \in R_{\mu, i, (x,y)}^{l}} \varepsilon_1 (x - x') + \varepsilon_2 (y - y')}{\displaystyle\prod_{(x', y') \in A_{\lambda, i, (x,y)}^{l}} \varepsilon_1 (x - x' + 1) + \varepsilon_2 (y - y' + 1)} \dfrac{\displaystyle\prod_{(x', y') \in R_{\mu, i, (x,y)}^{r}} \varepsilon_1 (x - x' - 1) + \varepsilon_2 (y - y' - 1)}{\displaystyle\prod_{(x', y') \in A_{\lambda, i, (x,y)}^{r}} \varepsilon_1 (x - x') + \varepsilon_2 (y - y')},
\end{align*}
\item
\begin{align*}
&\dfrac{\langle P_{\lambda},  P_{\lambda} \rangle_F}{\langle P_{\mu},  P_{\mu} \rangle_F} \\
& = \prod_{(x', y') \in A_{\lambda, i, (x,y)}^{l}} \dfrac{\varepsilon_1 (x - x') + \varepsilon_2 (y - y')}{\varepsilon_1 (x - x' + 1) + \varepsilon_2 (y - y' + 1)} \prod_{(x', y') \in A_{\lambda, i, (x,y)}^{r}} \dfrac{\varepsilon_1 (x - x' + 1) + \varepsilon_2 (y - y' + 1)}{\varepsilon_1 (x - x') + \varepsilon_2 (y - y')} \\
& \prod_{(x', y') \in R_{\mu, i, (x,y)}^{l}} \dfrac{\varepsilon_1 (x - x') + \varepsilon_2 (y - y')}{\varepsilon_1 (x - x' - 1) + \varepsilon_2 (y - y' - 1)} \prod_{(x', y') \in R_{\mu, i, (x,y)}^{r}} \dfrac{\varepsilon_1 (x - x' - 1) + \varepsilon_2 (y - y' - 1)}{\varepsilon_1 (x - x') + \varepsilon_2 (y - y')}.
\end{align*}
\end{enumerate}
\end{lem}

\begin{proof}
Let us prove (i).
All factors cancel except concerning with cells
\begin{enumerate}
\item[(L)] $s = (x', y)$ for $0 \leq x' < x$; or
\item[(R)] $s = (x, y')$ for $0 \leq y' < y$.
\end{enumerate}
In the case (L) we have
\[
l_{\lambda}(s) = x - x',\ l_{\mu}(s) = x - x' - 1,
\]
\[
a_{\lambda}(s) = a_{\mu}(s) = \lambda_{x'+1} - (y+1).
\]
If we put $y' = \lambda_{x'+1}$ so that $h_{\lambda}(s) \equiv 0 \bmod N$ is equivalent to that $(x', y')$ is an $i$-cell, then the numerator is
\[
\prod_{\substack{s = (x', y) \in \lambda\\ h_{\lambda}(s) \equiv 0 \bmod N\\ 0 \leq x' < x}} \varepsilon_1 l_{\lambda}(s) - \varepsilon_2 (a_{\lambda}(s) + 1) = \prod_{0 \leq x' < x} \varepsilon_1 (x -x ') - \varepsilon_2 (y' - y).
\]
If we put $y' = \lambda_{x'+1} - 1$ so that $h_{\mu}(s) \equiv 0 \bmod N$ is equivalent to that $(x', y')$ is an $i$-cell, then the denominator is
\[
\displaystyle\prod_{\substack{s = (x', y) \in \mu\\ h_{\mu}(s) \equiv 0 \bmod N\\ 0 \leq x' < x}} \varepsilon_1 l_{\mu}(s) - \varepsilon_2 (a_{\mu}(s) + 1) = \prod_{0 \leq x' < x} \varepsilon_1 (x - x ' - 1) - \varepsilon_2 (y' - y + 1).
\]
Then their ratio is given by
\[
\dfrac{\displaystyle\prod_{(x', y') \in A_{\lambda, i, (x,y)}^{l}} \varepsilon_1 (x - x') + \varepsilon_2 (y - y')}{\displaystyle\prod_{(x', y') \in R_{\mu, i, (x,y)}^{l}} \varepsilon_1 (x - x' - 1) + \varepsilon_2 (y - y' - 1)}
\]
after cancellations.
In the case (R) we have
\[
l_{\lambda}(s) = l_{\mu}(s) = {}^t\lambda_{y'+1} - (x+1),
\]
\[
a_{\lambda}(s) = y - y',\  a_{\mu}(s) = y - y' - 1.
\]
If we put $x' = {}^t\lambda_{y'+1}$ so that $h_{\lambda}(s) \equiv 0 \bmod N$ is equivalent to that $(x', y')$ is an $i$-cell, then the numerator is
\[
\prod_{\substack{s = (x,y') \in \lambda\\ h_{\lambda}(s) \equiv 0 \bmod N\\ 0 \leq y' < y}} \varepsilon_1 l_{\lambda}(s) - \varepsilon_2 (a_{\lambda}(s) + 1) = \prod_{0 \leq y' < y} \varepsilon_1 (x ' - x - 1) - \varepsilon_2 (y - y' + 1).
\]
If we put $x' = {}^t\lambda_{y'+1} - 1$ so that $h_{\mu}(s) \equiv 0 \bmod N$ is equivalent to that $(x', y')$ is an $i$-cell, then the denominator is
\[
\displaystyle\prod_{\substack{s = (x,y') \in \mu\\ h_{\mu}(s) \equiv 0 \bmod N\\ 0 \leq y' < y}} \varepsilon_1 l_{\mu}(s) - \varepsilon_2 (a_{\mu}(s) + 1) = \prod_{0 \leq y' < y} \varepsilon_1 (x ' - x) - \varepsilon_2 (y - y').
\]
Then their ratio is
\[
(-1)^{\# A_{\lambda, i, (x,y)}^{r} - \# R_{\mu, i, (x,y)}^{r}} \dfrac{\displaystyle\prod_{(x', y') \in A_{\lambda, i, (x,y)}^{r}} \varepsilon_1 (x - x' + 1) + \varepsilon_2 (y - y' + 1)}{\displaystyle\prod_{(x', y') \in R_{\mu, i, (x,y)}^{r}} \varepsilon_1 (x - x') + \varepsilon_2 (y - y')}.
\]
The proof of (i) is complete.

We prove (ii) by a similar argument.
In the case (L), if we put $y' = \lambda_{x'+1} - 1$ so that $h_{\mu}(s) \equiv 0 \bmod N$ is equivalent to that $(x', y')$ is an $i$-cell, then the numerator is
\[
\displaystyle\prod_{\substack{s = (x', y) \in \mu\\ h_{\mu}(s) \equiv 0 \bmod N\\ 0 \leq x' < x}} \varepsilon_1 (l_{\mu}(s) + 1) - \varepsilon_2 a_{\mu}(s) = \prod_{0 \leq x' < x} \varepsilon_1 (x - x ') - \varepsilon_2 (y' - y).
\]
If we put $y' = \lambda_{x'+1}$ so that $h_{\lambda}(s) \equiv 0 \bmod N$ is equivalent to that $(x', y')$ is an $i$-cell, then the denominator is
\[
\prod_{\substack{s = (x',y) \in \lambda\\ h_{\lambda}(s) \equiv 0 \bmod N\\ 0 \leq x' < x}} \varepsilon_1 (l_{\lambda}(s) + 1) - \varepsilon_2 a_{\lambda}(s) = \prod_{0 \leq x' < x} \varepsilon_1 (x - x ' + 1) - \varepsilon_2 (y' - y - 1).
\]
Then their ratio is
\[
\dfrac{\displaystyle\prod_{(x', y') \in R_{\mu, i, (x,y)}^{l}} \varepsilon_1 (x - x') + \varepsilon_2 (y - y')}{\displaystyle\prod_{(x', y') \in A_{\lambda, i, (x,y)}^{l}} \varepsilon_1 (x - x' + 1) + \varepsilon_2 (y - y' + 1)}.
\]
In the case (R), if we put $x' = {}^t\lambda_{y'+1} - 1$ so that $h_{\mu}(s) \equiv 0 \bmod N$ is equivalent to that $(x', y')$ is an $i$-cell, then the numerator is
\[
\displaystyle\prod_{\substack{s = (x, y') \in \mu\\ h_{\mu}(s) \equiv 0 \bmod N\\ 0 \leq y' < y}} \varepsilon_1 (l_{\mu}(s) + 1) - \varepsilon_2 a_{\mu}(s) = \prod_{0 \leq y' < y} \varepsilon_1 (x ' - x + 1) - \varepsilon_2 (y - y' - 1).
\]
If we put $x' = {}^t\lambda_{y'+1}$ so that $h_{\lambda}(s) \equiv 0 \bmod N$ is equivalent to that $(x', y')$ is an $i$-cell, then the denominator is
\[
\prod_{\substack{s = (x, y') \in \lambda\\ h_{\lambda}(s) \equiv 0 \bmod N\\ 0 \leq y' < y}} \varepsilon_1 (l_{\lambda}(s) + 1) - \varepsilon_2 a_{\lambda}(s) = \prod_{0 \leq y' < y} \varepsilon_1 (x ' - x) - \varepsilon_2 (y - y').
\]
Then their ratio is
\[
(-1)^{\# A_{\lambda, i, (x,y)}^{r} - \# R_{\mu, i, (x,y)}^{r}} \dfrac{\displaystyle\prod_{(x', y') \in R_{\mu, i, (x,y)}^{r}} \varepsilon_1 (x - x' - 1) + \varepsilon_2 (y - y' - 1)}{\displaystyle\prod_{(x', y') \in A_{\lambda, i, (x,y)}^{r}} \varepsilon_1 (x - x') + \varepsilon_2 (y - y')}.
\]
The proof of (ii) is complete.

The assertion (iii) follows from (i) and (ii).
\end{proof}

\section{Explicit formula}\label{sec:explicit}

\subsection{Gelfand-Zetlin basis}

Recall the set $M$ defined in Subsection~\ref{subsection:preliminaries}.
We associate $r_1, \ldots, r_l$ and $p_1, \ldots, p_l$ with an element ${\bf m} \in M$. 

Let $\mathcal{S}_{\bf m}$ be the set which consists of collections of integers $\Lambda = (\lambda_{i,p}^{(s)})_{i,p,s}$, where indices run through $i = 1, \ldots, N,$ $p = 1, \ldots, i,$ $s = 1, \ldots, l$, satisfying:
\[
\lambda_{N,p}^{(s)} = 
\begin{cases}
1 \text{ if } 1 \leq p \leq p_s, \\
0 \text{ otherwise,} 
\end{cases}
\]
\[
\lambda_{i,p}^{(s)} \geq \lambda_{i-1,p}^{(s)} \geq \lambda_{i,p+1}^{(s)}.
\]
Although the conditions depend only on the data $p_1, \ldots, p_l$, we distinguish them for different ${\bf m}$.

For each $s$, a collection $\Lambda^{(s)} = (\lambda_{i,p}^{(s)})_{i,p}$ of integers is called a Gelfand-Zetlin scheme.
In fact, this is a restricted class of usual Gelfand-Zetlin schemes.
A Gelfand-Zetlin scheme used by Nazarov-Tarasov~\cite[Section~2]{MR1605817} is defined for any dominant integral weight $\lambda$ of $\mathfrak{gl}_N$.
The above definition corresponds to the fundamental weight $\varpi_{p_s}$.

Let $\Lambda = (\lambda_{i,p}^{(s)}) \in \mathcal{S}_{\bf m}$ be a tuple of Gelfand-Zetlin schemes associated with ${\bf m} \in M$.
For a fixed $(s,i)$ there exists a unique $p$ such that $\lambda_{i,p}^{(s)} = 1$ and $\lambda_{i,p+1}^{(s)} = 0$.
We write $l_{i}^{(s)}$ for such $p$.

The $Y_{\hbar'}(\mathfrak{gl}_N)$-module $V(\varpi_{p_1})_{a_1} \otimes \cdots \otimes V(\varpi_{p_l})_{a_l}$ with
\[
a_s = t r_s + \hbar'(p_1 + \cdots + p_s - \dfrac{3}{2})
\]
has a basis $\{\xi_{\Lambda}\}_{\Lambda \in \mathcal{S}_{\bf m}}$, the Gelfand-Zetlin basis constructed by Nazarov-Tarasov~\cite[Section~3]{MR1605817}.
Each $\xi_{\Lambda}$ is a simultaneous eigenvector of $A_{\hbar'}(\mathfrak{gl}_N)$ and explicit formulas for the actions of $A_i(u), B_i(u), C_i(u)$ are known.

The element of $\mathcal{S}_{\bf m}$ corresponding to the highest weight vector is $(\kappa_{i,p}^{(s)})$, where
\[
\kappa_{i,p}^{(s)} =
\begin{cases}
1 \text{ if } 1 \leq p \leq \min\{i, p_s\}, \\
0 \text{ otherwise.} 
\end{cases}
\]
We have $\kappa_{i,p}^{(s)} \geq \lambda_{i,p}^{(s)}$ for every $(s,i,p)$ and $\Lambda = (\lambda_{i,p}^{(s)}) \in \mathcal{S}_{\bf m}$.
For $\Lambda \in \mathcal{S}_{\bf m}$ and $(s,i,p)$, $\Lambda \pm \delta_{i,p}^{(s)}$ denotes the collection of integers whose entries are the same as in $\Lambda$ except for the $(s,i,p)$-entry $\lambda_{i,p}^{(s)} \pm 1$.

For a fixed $i = 1, \ldots, N-1$ and $\Lambda \in \mathcal{S}_{\bf m}$, we define 
\begin{align*}
R_{\Lambda, i} &= \{ (s,p) \mid \Lambda + \delta_{i,p}^{(s)} \in \mathcal{S}_{\bf m}\}, \\
A_{\Lambda, i} &= \{ (s,p) \mid \Lambda - \delta_{i,p}^{(s)} \in \mathcal{S}_{\bf m}\}.
\end{align*}
Set
\[
\nu_{i,p}^{(s)} = \hbar'(p - 1 - \lambda_{i,p}^{(s)}) - a_s,
\]
where
\[
a_s = t r_s + \hbar'(p_1 + \cdots + p_s - \dfrac{3}{2}).
\]
Fix $(s,i,p)$ and set
\begin{align*}
K_{+,p'}^{(s')} &= \hbar'( p - p' + \kappa_{i+1, p'}^{(s')} ) - (a_s - a_{s'}),  \\
L_{+,p'}^{(s')} &= \hbar'( p - p' + \lambda_{i+1, p'}^{(s')} ) - (a_s - a_{s'}), \\
K_{-,p'}^{(s')} &= \hbar'( p - p' - 1 + \kappa_{i-1, p'}^{(s')}) - (a_s - a_{s'}), \\
L_{-,p'}^{(s')} &= \hbar'( p - p' - 1 + \lambda_{i-1, p'}^{(s')} ) - (a_s - a_{s'})
\end{align*}
for each $(s', p')$.
Put
\begin{align*}
\gamma_{i,p}^{(s)} &= \prod_{s'=1}^{l} \left( \prod_{p'=1}^{p} K_{+,p'}^{(s')} \prod_{p'=p+1}^{i+1} L_{+,p'}^{(s')} \prod_{p'=1}^{p-1} K_{-,p'}^{(s')} \prod_{p'=p}^{i-1} L_{-,p'}^{(s')} \right), \\
\beta_{i,p}^{(s)} &= \prod_{s'=1}^{l} \left( \prod_{p'=1}^{p} \dfrac{L_{+,p'}^{(s')}}{K_{+,p'}^{(s')}} \prod_{p'=1}^{p-1} \dfrac{L_{-,p'}^{(s')}}{K_{-,p'}^{(s')}} \right).
\end{align*}
These are introduced in \cite[p.\ 203 and p.\ 205]{MR1605817}.
In our case, note that $\lambda_{i,p}^{(s)} = 0$ if $(s,p) \in R_{\Lambda,i}$ and $\lambda_{i,p}^{(s)} = 1$ if $(s,p) \in A_{\Lambda,i}$.
Also note that the values $M^{(s)}$ in \cite{MR1605817} are all zero.
Hence we have these forms of $\gamma_{i,p}^{(s)}, \beta_{i,p}^{(s)}$.
Then the formulas given by Nazarov-Tarasov~\cite[Theorem~3.2, Theorem~3.5, Theorem~3.8]{MR1605817} read as
\begin{equation}
\left( \prod_{s=1}^{l}\prod_{p=1}^{i} (u - \hbar'(p - 1) + a_s) \right) A_i(u) \xi_{\Lambda} = \left( \prod_{s=1}^{l}\prod_{p=1}^{i} (u - \nu_{i,p}^{(s)}) \right) \xi_{\Lambda}, \label{eq:A_i(u)}
\end{equation}
\begin{equation}
\left( \prod_{s=1}^{l}\prod_{p=1}^{i} (u - \hbar'(p - 1) + a_s) \right) B_i(u) \xi_{\Lambda} = \sum_{(s,p) \in R_{\Lambda,i}}(-\gamma_{i,p}^{(s)}) \left( \prod_{(s',p') \neq (s,p)} \dfrac{u - \nu_{i,p'}^{(s')}}{\nu_{i,p}^{(s)} - \nu_{i,p'}^{(s')}} \right) \xi_{\Lambda + \delta_{i,p}^{(s)}}, \label{eq:B_i(u)} \\
\end{equation}
\begin{equation}
\left( \prod_{s=1}^{l}\prod_{p=1}^{i} (u - \hbar'(p - 1) + a_s) \right) C_i(u) \xi_{\Lambda} = \sum_{(s,p) \in A_{\Lambda,i}} \beta_{i,p}^{(s)} \left( \prod_{(s',p') \neq (s,p)} \dfrac{u - \nu_{i,p'}^{(s')}}{\nu_{i,p}^{(s)} - \nu_{i,p'}^{(s')}} \right) \xi_{\Lambda - \delta_{i,p}^{(s)}}. \label{eq:C_i(u)}
\end{equation}
Here we use the Lagrange interpolation to derive (\ref{eq:B_i(u)}), (\ref{eq:C_i(u)}) from the formulas in \cite{MR1605817} thanks to the fact that the left hand sides of (\ref{eq:B_i(u)}) and (\ref{eq:C_i(u)}) are polynomials in $u$ whose degree do not exceed $il - 1$ \cite[Proposition 3.1]{MR1605817}.

\begin{rem}
We need to swap $B_i(u)$ and $C_i(u)$ in the formulas of \cite{MR1605817} since the generator $T_{ij}^{(r)}$ in \cite{MR1605817} corresponds to our $T_{ji}^{(r)}$.
\end{rem}

\subsection{Gelfand-Zetlin schemes and partitions}

We give an identification between Gelfand-Zetlin schemes and partitions.
Let $\Lambda \in \mathcal{S}_{\bf m}$ be a tuple of Gelfand-Zetlin schemes.
We associate a sequence $(i_p^{(s)})$ with $\Lambda=(\lambda_{i,p}^{(s)})$ by
\[
i_p^{(s)} = \min\{i \mid \lambda_{i,p}^{(s)} \neq 0\}
\]
for $p=1,\ldots,p_s$.
We can recover $\Lambda$ by
\[
\lambda_{i,p}^{(s)} = \delta(i_{p}^{(s)} \leq i),
\]
where $\delta(P)$ takes $1$ if $P$ is true and $0$ if $P$ is false.
Note that $\lambda_{i,p}^{(s)} = 0$ if $p > p_s$ by definition.
Set $n = p_1 + \cdots +p_l$.
Define $j_a$ for $a = 1, \ldots, n$ by
\[
j_{p_1 + \cdots + p_s - p +1} = i_p^{(s)}
\]
for $p=1, \ldots, p_s$ and $s = 1, \ldots l$, that is, 
\[
(j_1, j_2, \ldots, j_n) = (i_{p_1}^{(1)}, \ldots, i_{1}^{(1)}, i_{p_2}^{(2)}, \ldots, i_{1}^{(2)},\ldots, i_{p_l}^{(l)}, \ldots, i_{1}^{(l)})
\]
and assign a partition $\lambda$ so that
\[
\lambda_a - a + 1 = j_a - N m_a.
\]
This gives a one-to-one correspondence between $\mathcal{S}_{\bf m}$ and $\mathcal{P}_{\bf m}$.
Note that we have $l(\lambda) \leq n$.

We can see that the subsets $R_{\Lambda,i}$ and $A_{\Lambda, i}$ correspond to $R_{\lambda,i}$ and $A_{\lambda,i}$ respectively, under this identification of $\Lambda$ and $\lambda$ as follows.
Suppose $(s,p) \in R_{\Lambda,i}$, then we have $i_{p}^{(s)} = i+1$.
If we put $x = p_1 + \cdots + p_s - p$ and $y=\lambda_{x+1} - 1$ then $(x, y) \in R_{\lambda,i}$.
Moreover $\Lambda + \delta_{i,p}^{(s)}$ corresponds to the partition obtained from $\lambda$ by removing the $i$-cell $(x,y)$.
Similarly suppose $(s,p) \in A_{\Lambda,i}$.
Then we have $i_{p}^{(s)} = i$.
If we put $x = p_1 + \cdots + p_s - p$ and $y=\lambda_{x+1}$ then $(x, y) \in A_{\lambda,i}$.
Moreover $\Lambda - \delta_{i,p}^{(s)}$ corresponds to the partition obtained from $\lambda$ by adding the $i$-cell $(x,y)$.
We remark that in the case $l(\lambda) = n$, the residue of $(l(\lambda), 0)$ is equal to $0$ since $n$ is a multiple of $N$.
Therefore we have $(l(\lambda), 0) \in A_{\lambda,0}$ and $\lambda$ does not correspond to any element of $A_{\Lambda,i}$ for $i=1, \ldots, N-1$.

We sometimes write $\xi_{\lambda}$ instead of $\xi_{\Lambda}$ under this identification.

\subsection{Explicit formula}

In the sequel, we regard the basis $\{\xi_{\lambda}\}_{\lambda \in \mathcal{P}_{\bf m}}$ as elements of $F_{\bf m} \cong \omega_{f(u; r)}^* ( V(\varpi_{p_1})_{a_1} \otimes \cdots \otimes V(\varpi_{p_l})_{a_l} )$ via pullback.
This does not affect the $Y_{\hbar}(\mathfrak{sl}_N)$-module structure.

First we show that the eigenvalue of $A_i(u)$ for $P_{\lambda}$ and that for $\xi_{\lambda}$ coincide, and then calculate $H_i(u) P_{\lambda}$.
We rewrite the formula (\ref{eq:A_i(u)}) for $A_i(u)$:
\begin{align*}
A_i(u) \xi_{\lambda} 
= g_i(u) \left( \prod_{s=1}^{l}\prod_{p=1}^{i} \dfrac{u - \hbar'(p - 1 - \lambda_{i,p}^{(s)}) + a_s}{u - \hbar'(p - 1) + a_s} \right) \xi_{\lambda}.
\end{align*}
We have
\begin{align*}
\prod_{s=1}^{l}\prod_{p=1}^{i} \dfrac{u - \hbar'(p - 1 - \lambda_{i,p}^{(s)}) + a_s}{u - \hbar'(p - 1) + a_s} = \prod_{s=1}^{l}\prod_{p=1}^{p_s} \dfrac{u - \hbar'(p - 1 - \delta( i_{p}^{(s)} \leq i )) + a_s}{u - \hbar'(p - 1) + a_s}
\end{align*}
since
\[
\lambda_{i,p}^{(s)} = \begin{cases}
\delta( i_{p}^{(s)} \leq i ) \text{ if } 1 \leq p \leq p_s, \\
0 \text{ if } p > p_s.
\end{cases}
\]
We have
\begin{align}
-\hbar'(p - 1) + a_s &= t r_s + \hbar'(p_1 + \cdots + p_s - p + 1 - \dfrac{3}{2}) \nonumber \\
&= t m_{p_1 + \cdots + p_s - p +1} + \hbar'(p_1 + \cdots + p_s - p + 1 - \dfrac{3}{2}) \label{eq:3.3.1}
\end{align}
for $p= 1, \ldots, p_s$.
If we put
\[
a = p_1 + \cdots + p_s - p + 1
\]
and vary $p = 1, \ldots, p_s$, $s=1, \ldots, l$, then $a$ runs through $1, \ldots, n$.
Hence we have
\[
\prod_{s=1}^{l}\prod_{p=1}^{p_s} \dfrac{u - \hbar'(p - 1 - \delta( i_{p}^{(s)} \leq i )) + a_s}{u - \hbar'(p - 1) + a_s} = \prod_{a=1}^n \dfrac{u + tm_a + \hbar'(a - \dfrac{3}{2} + \delta( j_a \leq i ))}{u + t m_a + \hbar'(a - \dfrac{3}{2})}. 
\]
Comparing with (\ref{eq:eigenvalue}), we conclude that $\xi_{\lambda}$ is a constant multiple of $P_{\lambda}$ since both $\{P_{\lambda}\}_{\lambda \in \mathcal{P}}$ and $\{\xi_{\lambda}\}_{\lambda \in \mathcal{P}}$ are eigenbases with pairwise distinct eigenvalues.
Hence we have
\[
P_{\lambda} = \alpha_{\lambda} \xi_{\lambda}
\]
with some nonzero $\alpha_{\lambda} \in \mathbb{C}(\varepsilon_1, \varepsilon_2)$.
Recall the formula (\ref{eq:H_{i,r}}):
\[
H_i(u) = \dfrac{A_{i-1}(u - \dfrac{1}{2}\hbar(i-1)) A_{i+1}(u - \dfrac{1}{2}\hbar(i+1))}{A_i(u - \dfrac{1}{2}\hbar(i-1)) A_i(u - \dfrac{1}{2}\hbar(i+1))}.
\]
The eigenvalue of
\[
\dfrac{A_{i-1}(u - \dfrac{1}{2}\hbar(i-1))}{A_i(u - \dfrac{1}{2}\hbar(i-1))}
\] is given by
\begin{equation}
 \dfrac{g_{i-1}(u - \dfrac{1}{2}\hbar(i-1))}{g_i(u - \dfrac{1}{2}\hbar(i-1))} \prod_{a=1}^n \dfrac{u + tm_a - \hbar(a - \dfrac{3}{2} + \delta( j_a \leq i-1 )) - \dfrac{1}{2}\hbar(i-1)}{u + tm_a - \hbar(a - \dfrac{3}{2} + \delta( j_a \leq i)) - \dfrac{1}{2}\hbar(i-1)}. \label{eq:factor1}
\end{equation}
We have
\begin{align*}
& \prod_{a=1}^n \dfrac{u + tm_a - \hbar(a - \dfrac{3}{2} + \delta( j_a \leq i-1 )) - \dfrac{1}{2}\hbar(i-1)}{u + tm_a - \hbar(a - \dfrac{3}{2} + \delta( j_a \leq i)) - \dfrac{1}{2}\hbar(i-1)} \\
&= \prod_{j_a = i} \dfrac{u + tm_a - \hbar(a - \dfrac{3}{2}) - \dfrac{1}{2}\hbar(i-1)}{u + tm_a - \hbar(a - \dfrac{3}{2} + 1) - \dfrac{1}{2}\hbar(i-1)} \\
&= \prod_{j_a = i} \dfrac{u + tm_a - \hbar(a + \dfrac{1}{2}i - 2)}{u + tm_a - \hbar(a + \dfrac{1}{2}i - 1)} \\
&= \prod_{(a-1, \lambda_a) \in A_{\lambda,i}} \dfrac{u + tm_a - \hbar(a + \dfrac{1}{2}i - 2)}{u + tm_a - \hbar(a + \dfrac{1}{2}i - 1)} \prod_{\substack{ m_{a-1} = m_a \\ j_{a-1} = i+1 \\ j_a = i}} \dfrac{u + tm_a - \hbar(a + \dfrac{1}{2}i - 2)}{u + tm_a - \hbar(a + \dfrac{1}{2}i - 1)}.
\end{align*}
The last equality follows from (\ref{eq:SetA}), the description of $A_{\lambda, i}$.
Recall that if $\lambda$ satisfies $l(\lambda) = n$, then $(l(\lambda), 0) \in A_{\lambda,0}$ holds, hence it does not appear in the product.
Similarly the eigenvalue of
\[
\dfrac{A_{i+1}(u - \dfrac{1}{2}\hbar(i+1))}{A_i(u - \dfrac{1}{2}\hbar(i+1))}
\]
is given by
\begin{equation}
\dfrac{g_{i+1}(u - \dfrac{1}{2}\hbar(i+1))}{g_i(u - \dfrac{1}{2}\hbar(i+1))} \prod_{a = 1}^n \dfrac{u + tm_a - \hbar(a - \dfrac{3}{2} + \delta( j_a \leq i+1 )) - \dfrac{1}{2}\hbar(i+1)}{u + tm_a - \hbar(a - \dfrac{3}{2} + \delta( j_a \leq i)) - \dfrac{1}{2}\hbar(i+1)}. \label{eq:factor2}
\end{equation}
We have
\begin{align*}
& \prod_{a = 1}^n \dfrac{u + tm_a - \hbar(a - \dfrac{3}{2} + \delta( j_a \leq i+1 )) - \dfrac{1}{2}\hbar(i+1)}{u + tm_a - \hbar(a - \dfrac{3}{2} + \delta( j_a \leq i)) - \dfrac{1}{2}\hbar(i+1)} \\
&= \prod_{j_a = i+1} \dfrac{u + tm_a - \hbar(a - \dfrac{3}{2} + 1) - \dfrac{1}{2}\hbar(i+1)}{u + tm_a - \hbar(a - \dfrac{3}{2}) - \dfrac{1}{2}\hbar(i+1)} \\
&= \prod_{j_a = i+1} \dfrac{u + tm_a - \hbar(a + \dfrac{1}{2}i)}{u + tm_a - \hbar(a + \dfrac{1}{2}i - 1)} \\
&= \prod_{(a-1, \lambda_a - 1) \in R_{\lambda,i}} \dfrac{u + tm_a - \hbar(a + \dfrac{1}{2}i)}{u + tm_a - \hbar(a + \dfrac{1}{2}i - 1)} \prod_{\substack{ m_{a} = m_{a+1} \\ j_{a} = i+1 \\ j_{a+1} = i}} \dfrac{u + tm_a - \hbar(a + \dfrac{1}{2}i)}{u + tm_a - \hbar(a + \dfrac{1}{2}i - 1)}.
\end{align*}
The last equality follows from (\ref{eq:SetR}), the description of $R_{\lambda, i}$.
Consider the product of (\ref{eq:factor1}) and (\ref{eq:factor2}).
We have
\[
\dfrac{g_{i-1}(u - \dfrac{1}{2}\hbar(i-1)) g_{i+1}(u - \dfrac{1}{2}\hbar(i+1))}{g_i(u - \dfrac{1}{2}\hbar(i-1)) g_i(u - \dfrac{1}{2}\hbar(i+1))} = 1
\]
and
\[
\prod_{\substack{ m_{a-1} = m_a \\ j_{a-1} = i+1 \\ j_a = i}} \dfrac{u + tm_a - \hbar(a + \dfrac{1}{2}i - 2)}{u + tm_a - \hbar(a + \dfrac{1}{2}i - 1)} \prod_{\substack{ m_{a} = m_{a+1} \\ j_{a} = i+1 \\ j_{a+1} = i}} \dfrac{u + tm_a - \hbar(a + \dfrac{1}{2}i)}{u + tm_a - \hbar(a + \dfrac{1}{2}i - 1)} = 1.
\]
In both cases when $(x,y) = (a-1, \lambda_a) \in A_{\lambda,i}$ and $(x,y) = (a-1, \lambda_a - 1) \in R_{\lambda,i}$, we have
\begin{align}
tm_a &= N \varepsilon_2 m_a \nonumber \\
&= \varepsilon_2 (x - y + i) \label{eq:3.3.2}
\end{align}
by (\ref{eq:(x,y)}).
Therefore we have
\begin{align*}
& H_i(u) P_{\lambda}\\
&= \prod_{(x,y) \in A_{\lambda,i}}\dfrac{u - \varepsilon_1(x + \dfrac{1}{2}i - 1) - \varepsilon_2(y - \dfrac{1}{2}i - 1)}{u - \varepsilon_1(x + \dfrac{1}{2}i) - \varepsilon_2(y - \dfrac{1}{2}i)} \prod_{(x,y) \in R_{\lambda,i}}\dfrac{u - \varepsilon_1(x + \dfrac{1}{2}i + 1) - \varepsilon_2(y - \dfrac{1}{2}i + 1)}{u - \varepsilon_1(x + \dfrac{1}{2}i) - \varepsilon_2(y - \dfrac{1}{2}i)} P_{\lambda}.
\end{align*}

Next we calculate $X_{i,r}^{\pm} P_{\lambda}$.
Let $\lambda, \mu$ be partitions and define $E_{\lambda \mu}^{(r)}, F_{\lambda \mu}^{(r)}$ by
\begin{align*}
X_{i,r}^{+} P_{\lambda} &= \sum_{\mu} E_{\lambda \mu}^{(r)} P_{\mu}, \\
X_{i,r}^{-} P_{\lambda} &= \sum_{\mu} F_{\lambda \mu}^{(r)} P_{\mu}.
\end{align*}
Similarly we define $\tilde{E}_{\lambda \mu}^{(r)}, \tilde{F}_{\lambda \mu}^{(r)}$ by
\begin{align*}
X_{i,r}^{+} \xi_{\lambda} &= \sum_{\mu} \tilde{E}_{\lambda \mu}^{(r)} \xi_{\mu}, \\
X_{i,r}^{-} \xi_{\lambda} &= \sum_{\mu} \tilde{F}_{\lambda \mu}^{(r)} \xi_{\mu}.\end{align*}
Then we have
\begin{equation}
E_{\lambda \mu}^{(r)} F_{\mu \lambda}^{(r)} = \tilde{E}_{\lambda \mu}^{(r)} \tilde{F}_{\mu \lambda}^{(r)} \label{eq:product}
\end{equation}
since $P_{\lambda}$ is a scalar multiple of $\xi_{\lambda}$.
By Proposition~\ref{prop:adjoint}, we have
\[
\langle X_{i,r}^{+} P_{\lambda}, P_{\mu} \rangle_{F} = \langle P_{\lambda}, X_{i,r}^{-} P_{\mu} \rangle_{F},
\]
and hence the equality
\[
E_{\lambda \mu}^{(r)} \langle P_{\mu}, P_{\mu} \rangle_{F} = F_{\mu \lambda}^{(r)} \langle P_{\lambda}, P_{\lambda} \rangle_{F}.
\]
Substituting (\ref{eq:product}) to this, we have
\[
\left( E_{\lambda \mu}^{(r)} \right)^2 = \tilde{E}_{\lambda \mu}^{(r)} \tilde{F}_{\mu \lambda}^{(r)} \dfrac{\langle P_{\lambda}, P_{\lambda} \rangle_{F}}{\langle P_{\mu}, P_{\mu} \rangle_{F}}.
\]
On the right-hand side, $\tilde{E}_{\lambda \mu}^{(r)} \tilde{F}_{\mu \lambda}^{(r)}$ can be calculated and $\dfrac{\langle P_{\lambda}, P_{\lambda} \rangle_{F}}{\langle P_{\mu}, P_{\mu} \rangle_{F}}$ has been calculated in Lemma~\ref{lem:ratio}.
Hence we can determine $E_{\lambda \mu}^{(r)}$ up to sign.
Our goal is to calculate $\tilde{E}_{\lambda \mu}^{(r)} \tilde{F}_{\mu \lambda}^{(r)}$ explicitly and to determine the sign.

By the formulas (\ref{eq:A_i(u)}), (\ref{eq:B_i(u)}), (\ref{eq:C_i(u)}), we have
\begin{align*}
A_i(u)^{-1} B_i(u) \xi_{\Lambda} &= \sum_{(s,p) \in R_{\Lambda,i}} \dfrac{-\gamma_{i,p}^{(s)}}{u - \hbar'(p-2) + a_s} \left( \prod_{(s',p') \neq (s,p)} \dfrac{1}{\nu_{i,p}^{(s)} - \nu_{i,p'}^{(s')}} \right) \xi_{\Lambda + \delta_{i,p}^{(s)}}, \\
C_i(u) A_i(u)^{-1} \xi_{\Lambda} &= \sum_{(s,p) \in A_{\Lambda,i}} \dfrac{\beta_{i,p}^{(s)}}{u - \hbar'(p-2) + a_s} \left( \prod_{(s',p') \neq (s,p)} \dfrac{1}{\nu_{i,p}^{(s)} - \nu_{i,p'}^{(s')}} \right) \xi_{\Lambda - \delta_{i,p}^{(s)}}.
\end{align*}
We have
\[
-\hbar'(p-2) + a_s = t m_a - \hbar(a- \dfrac{1}{2})
\]
by (\ref{eq:3.3.1}) where we put $a = p_1 + \cdots + p_s - p + 1$.
Let $(x,y) = (a - 1, \lambda_a - 1) \in R_{\lambda,i}$ be the cell corresponding to $(s,p)$ in the case $(s,p) \in R_{\Lambda,i}$, and $(x,y) = (a - 1, \lambda_a) \in A_{\lambda,i}$ in the case $(s,p) \in A_{\Lambda,i}$.
In both cases, we have
\[
-\hbar'(p-2) + a_s = \varepsilon_2(x-y+i) - \hbar(x + \dfrac{1}{2})
\]
by (\ref{eq:3.3.2}).

Recall the formulas (\ref{eq:X_{i,r}^+}) and (\ref{eq:X_{i,r}^-}):
\begin{align*}
-\hbar \sum_{r \geq 0} X_{i,r}^+ u^{-r-1} &= A_i(u - \dfrac{1}{2}\hbar(i-1))^{-1} B_i(u - \dfrac{1}{2}\hbar(i-1)), \\
-\hbar \sum_{r \geq 0} X_{i,r}^- u^{-r-1} &= C_i(u - \dfrac{1}{2}\hbar(i-1)) A_i(u - \dfrac{1}{2}\hbar(i-1))^{-1}. 
\end{align*}
Hence we see that 
\begin{align*}
\tilde{E}_{\lambda \mu}^{(r)} & = 
\begin{cases}
-\hbar^{-1} \left( \varepsilon_1 (x + \dfrac{1}{2}i) + \varepsilon_2 (y - \dfrac{1}{2}i) \right)^r  \left(  (-\gamma_{i,p}^{(s)}) \displaystyle\prod_{(s',p') \neq (s,p)} \dfrac{1}{\nu_{i,p}^{(s)} - \nu_{i,p'}^{(s')}} \right) \\
\quad \text{ if $\mu$ corresponds to $\Lambda + \delta_{i,p}^{(s)}$,} \\
0 \text{ otherwise,}
\end{cases}\\
\tilde{F}_{\lambda \mu}^{(r)} & = 
\begin{cases}
-\hbar^{-1} \left( \varepsilon_1 (x + \dfrac{1}{2}i) + \varepsilon_2 (y - \dfrac{1}{2}i) \right)^r  \left( \beta_{i,p}^{(s)} \displaystyle\prod_{(s',p') \neq (s,p)} \dfrac{1}{\nu_{i,p}^{(s)} - \nu_{i,p'}^{(s')}} \right) \\
\quad \text{ if $\mu$ corresponds to $\Lambda - \delta_{i,p}^{(s)}$,} \\
0 \text{ otherwise.}
\end{cases}
\end{align*}
Set
\begin{align*}
\tilde{E}_{\lambda \mu} &= -\hbar^{-1} (-\gamma_{i,p}^{(s)}) \displaystyle\prod_{(s',p') \neq (s,p)} \dfrac{1}{\nu_{i,p}^{(s)} - \nu_{i,p'}^{(s')}},\\
\tilde{F}_{\lambda \mu} &= -\hbar^{-1} \beta_{i,p}^{(s)} \displaystyle\prod_{(s',p') \neq (s,p)} \dfrac{1}{\nu_{i,p}^{(s)} - \nu_{i,p'}^{(s')}}.
\end{align*}
Note that $\gamma_{i,p}^{(s)}$ and $\beta_{i,p}^{(s)}$ can be written as
\begin{align*}
\gamma_{i,p}^{(s)} &= \prod_{s'=1}^{l} \left( \prod_{p'=1}^{p} K_{+,p'}^{(s')} \prod_{p'=p+1}^{i+1} L_{+,p'}^{(s')} \prod_{p'=1}^{p-1} K_{-,p'}^{(s')} \prod_{p'=p}^{i-1} L_{-,p'}^{(s')} \right) \\
&= \prod_{s'=1}^{l} \left( \prod_{p'=1}^{i+1} L_{+,p'}^{(s')} \prod_{p'=1}^{p} \dfrac{K_{+,p'}^{(s')}}{L_{+,p'}^{(s')}} \prod_{p'=1}^{p-1} K_{-,p'}^{(s')} \prod_{p'=p}^{i-1} L_{-,p'}^{(s')} \right)
\end{align*}
and
\begin{align*}
\beta_{i,p}^{(s)} &= \prod_{s'=1}^{l} \left( \prod_{p'=1}^{p} \dfrac{L_{+,p'}^{(s')}}{K_{+,p'}^{(s')}} \prod_{p'=1}^{p-1} \dfrac{L_{-,p'}^{(s')}}{K_{-,p'}^{(s')}} \right) \\
&= \prod_{s'=1}^{l} \left( \prod_{p'=1}^{i-1} L_{-,p'}^{(s')} \prod_{p'=1}^{p} \dfrac{L_{+,p'}^{(s')}}{K_{+,p'}^{(s')}} \prod_{p'=1}^{p-1} \dfrac{1}{K_{-,p'}^{(s')}} \prod_{p'=p}^{i-1} \dfrac{1}{L_{-,p'}^{(s')}} \right).
\end{align*}
We calculate
\[
\left( \prod_{s'=1}^{l} \prod_{p'=1}^{i+1} L_{+,p'}^{(s')} \right) \left( \prod_{(s',p') \neq (s,p)} \dfrac{1}{\nu_{i,p}^{(s)} - \nu_{i,p'}^{(s')}} \right)
= \left( \dfrac{\displaystyle\prod_{p'=1}^{i+1} L_{+,p'}^{(s)}}{\displaystyle\prod_{\substack{1 \leq p'\leq i\\ p' \neq p}} (\nu_{i,p}^{(s)} - \nu_{i,p'}^{(s)})} \right) \left( \prod_{s' \neq s} \dfrac{\displaystyle\prod_{p'=1}^{i+1} L_{+,p'}^{(s')}}{\displaystyle\prod_{p'=1}^{i} (\nu_{i,p}^{(s)} - \nu_{i,p'}^{(s')})} \right)
\]
and
\[
\left( \prod_{s'=1}^{l} \prod_{p'=1}^{i-1} L_{-,p'}^{(s')} \right) \left( \prod_{(s',p') \neq (s,p)} \dfrac{1}{\nu_{i,p}^{(s)} - \nu_{i,p'}^{(s')}} \right) = \left( \dfrac{\displaystyle\prod_{p'=1}^{i-1} L_{-,p'}^{(s)}}{\displaystyle\prod_{\substack{1 \leq p'\leq i\\ p' \neq p}} (\nu_{i,p}^{(s)} - \nu_{i,p'}^{(s)})} \right) \left( \prod_{s' \neq s} \dfrac{\displaystyle\prod_{p'=1}^{i-1} L_{-,p'}^{(s')}}{\displaystyle\prod_{p'=1}^{i} (\nu_{i,p}^{(s)} - \nu_{i,p'}^{(s')})} \right)
\]
in the next lemma.

\begin{lem}\label{lem:factor}
\begin{enumerate}
\item Let $(s,p) \in R_{\Lambda,i}$.
Then we have
\[
\dfrac{\displaystyle\prod_{p'=1}^{i+1} L_{+,p'}^{(s)}}{\displaystyle\prod_{\substack{1 \leq p'\leq i\\ p' \neq p}} (\nu_{i,p}^{(s)} - \nu_{i,p'}^{(s)})} = \hbar'^2 (p - i -1)
\]
and
\begin{align*}
& \dfrac{\displaystyle\prod_{p'=1}^{i+1} L_{+,p'}^{(s')}}{\displaystyle\prod_{p'=1}^{i} (\nu_{i,p}^{(s)} - \nu_{i,p'}^{(s')})}
&= \begin{cases}
\ \hbar'(p - i -1) - (a_s - a_{s'}) \text{ if } l_{i+1}^{(s')} = l_i^{(s')}, \\
\\
\ \hbar'(p - i) - (a_s - a_{s'}) \text{ if } l_{i+1}^{(s')} = l_i^{(s')}+1 \text{ and } l_i^{(s')} = i, \\
\\
\ \dfrac{\Big( \hbar'(p - l_{i}^{(s')}) - (a_s - a_{s'}) \Big) \Big( \hbar'(p - i -1) - (a_s - a_{s'}) \Big)}{\hbar'(p - l_{i}^{(s')} - 1) - (a_s - a_{s'})} \\
\\
\quad \text{ if } l_{i+1}^{(s')} = l_i^{(s')}+1 \text{ and } l_i^{(s')} \neq i
\end{cases}
\end{align*}
for $s' \neq s$.
\item Let $(s,p) \in A_{\Lambda,i}$.
Then we have
\[
\dfrac{\displaystyle\prod_{p'=1}^{i-1} L_{-,p'}^{(s)}}{\displaystyle\prod_{\substack{1 \leq p'\leq i\\ p' \neq p}} (\nu_{i,p}^{(s)} - \nu_{i,p'}^{(s)})} = -\dfrac{1}{p - i -1}
\]
and
\begin{align*}
& \dfrac{\displaystyle\prod_{p'=1}^{i-1} L_{-,p'}^{(s')}}{\displaystyle\prod_{p'=1}^{i} (\nu_{i,p}^{(s)} - \nu_{i,p'}^{(s')})} 
&= \begin{cases}
\ \dfrac{1}{\hbar'(p - i -1) - (a_s - a_{s'})} \text{ if } l_{i-1}^{(s')} = l_i^{(s')}, \\
\\
\ \dfrac{1}{\hbar'(p - i) - (a_s - a_{s'})} \text{ if } l_{i-1}^{(s')} = l_i^{(s')}-1 \text{ and } l_i^{(s')} = i, \\
\\
\ \dfrac{\hbar'(p - l_{i}^{(s')} - 1) - (a_s - a_{s'})}{\Big( \hbar'(p - l_{i}^{(s')}) - (a_s - a_{s'}) \Big) \Big( \hbar'(p - i -1) - (a_s - a_{s'}) \Big)} \\
\\
\quad \text{ if } l_{i-1}^{(s')} = l_i^{(s')}-1 \text{ and } l_i^{(s')} \neq i
\end{cases}
\end{align*}
for $s' \neq s$.
\end{enumerate}
\end{lem}

\begin{proof}
Let us prove (i).
We have $\lambda_{i,p}^{(s)} = 0$ by the assumption $(s,p) \in R_{\Lambda,i}$.
The ($s$)-factor is calculated as
\begin{align*}
\dfrac{\displaystyle\prod_{p'=1}^{i+1} L_{+,p'}^{(s)}}{\displaystyle\prod_{\substack{1 \leq p'\leq i\\ p' \neq p}} (\nu_{i,p}^{(s)} - \nu_{i,p'}^{(s)})} &= \dfrac{\displaystyle\prod_{p'=1}^{i+1} \hbar'( p - p' +  \lambda_{i+1, p'}^{(s)})}{\displaystyle\prod_{\substack{1 \leq p'\leq i\\ p' \neq p}} \hbar'( p - p' + \lambda_{i, p'}^{(s)})}\\
&= \dfrac{\displaystyle\prod_{p'=1}^{p} \hbar'(p - p' + 1) \displaystyle\prod_{p'=p+1}^{i+1} \hbar'(p - p')}{\displaystyle\prod_{p'=1}^{p-1} \hbar'(p - p' + 1) \displaystyle\prod_{p'=p+1}^{i} \hbar'(p - p')} \\
&= \hbar'^2 (p - i -1).
\end{align*}

Suppose  $s' \neq s$ and let us calculate
\begin{align*}
\dfrac{\displaystyle\prod_{p'=1}^{i+1} L_{+,p'}^{(s')}}{\displaystyle\prod_{p'=1}^{i} ( \nu_{i,p}^{(s)} - \nu_{i,p'}^{(s')}) } &= \dfrac{\displaystyle\prod_{p'=1}^{i+1} \Big( \hbar'(p - p' + \lambda_{i+1,p'}^{(s')}) - (a_s - a_{s'}) \Big)}{\displaystyle\prod_{p'=1}^{i} \Big( \hbar'(p - p' + \lambda_{i,p'}^{(s')}) - (a_s - a_{s'}) \Big)}.
\end{align*}
\begin{itemize}
\item If $l_{i+1}^{(s')} = l_i^{(s')} $ then we have $\lambda_{i+1,p'}^{(s')} = \lambda_{i,p'}^{(s')}$ for $p' = 1, \ldots, i$ and $\lambda_{i+1,i+1}^{(s')} = 0$.
Hence all but the ($p' = i+1$)-factor cancel and we obtain the result.
\item If $l_{i+1}^{(s')} = l_i^{(s')} + 1$ and $l_{i}^{(s')} = i$ then we have $\lambda_{i+1,p'}^{(s')} = \lambda_{i,p'}^{(s')}$ for $p' = 1, \ldots, i$ and $\lambda_{i+1,i+1}^{(s')} = 1$.
Hence all but the ($p' = i+1$)-factor cancel and we obtain the result.
\item If $l_{i+1}^{(s')} = l_i^{(s')} + 1$ and $l_{i}^{(s')} \neq i$ then we have $\lambda_{i+1,p'}^{(s')} = \lambda_{i,p'}^{(s')}$ except for $p' = l_i^{(s')} + 1, i+1$.
The result follows from $\lambda_{i+1,l_i^{(s')} + 1}^{(s')} = 1, \lambda_{i,l_i^{(s')} + 1}^{(s')} = 0$ and $\lambda_{i+1,i+1}^{(s')} = 0$.
\end{itemize}

Let us prove (ii).
We have $\lambda_{i,p}^{(s)} = 1$ by the assumption $(s,p) \in A_{\Lambda,i}$.
The ($s$)-factor is calculated as
\begin{align*}
\dfrac{\displaystyle\prod_{p'=1}^{i-1} L_{-,p'}^{(s)}}{\displaystyle\prod_{\substack{1 \leq p'\leq i\\ p' \neq p}} (\nu_{i,p}^{(s)} - \nu_{i,p'}^{(s)})} &= \dfrac{\displaystyle\prod_{p'=1}^{i-1} \hbar'( p - p' - 1 +  \lambda_{i-1, p'}^{(s)})}{\displaystyle\prod_{\substack{1 \leq p'\leq i\\ p' \neq p}} \hbar'( p - p' - 1 + \lambda_{i, p'}^{(s)})}\\
&= \dfrac{\displaystyle\prod_{p' = 1}^{p-1} \hbar'(p - p') \displaystyle\prod_{p' = p}^{i-1} \hbar'(p - p' - 1)}{\displaystyle\prod_{p' = 1}^{p-1} \hbar'(p - p') \displaystyle\prod_{p' = p+1}^i \hbar'(p - p' - 1) }.
\end{align*}
This is equal to
\[
\dfrac{-\hbar'}{\hbar'(p-i-1)} = - \dfrac{1}{p - i - 1}
\]
if $p \neq i$
and $1$ if $p = i$ for which the above is also true.

Suppose $s' \neq s$ and let us calculate
\begin{align*}
\dfrac{\displaystyle\prod_{p'=1}^{i-1} L_{-,p'}^{(s')}}{\displaystyle\prod_{p'=1}^{i} (\nu_{i,p}^{(s)} - \nu_{i,p'}^{(s')})} &= \dfrac{\displaystyle\prod_{p'=1}^{i-1} \Big( \hbar'(p - p' -1 + \lambda_{i-1,p'}^{(s')}) - (a_s - a_{s'}) \Big) }{\displaystyle\prod_{p'=1}^{i} \Big( \hbar'(p - p' -1 + \lambda_{i,p'}^{(s')}) - (a_s - a_{s'}) \Big) }.
\end{align*}
\begin{itemize}
\item If $l_{i-1}^{(s')} = l_i^{(s')}$ then we have $\lambda_{i-1,p'}^{(s')} = \lambda_{i,p'}^{(s')}$ for $p' = 1, \ldots, i-1$ and $\lambda_{i,i}^{(s')} = 0$.
Hence all but the ($p' = i$)-factor cancel and we obtain the result.
\item If $l_{i-1}^{(s')} = l_i^{(s')} - 1$ and $l_{i}^{(s')} = i$ then we have $\lambda_{i-1,p'}^{(s')} = \lambda_{i,p'}^{(s')}$ for $p' = 1, \ldots, i-1$ and $\lambda_{i,i}^{(s')} = 1$.
Hence all but the ($p' = i$)-factor cancel and we obtain the result.
\item If $l_{i-1}^{(s')} = l_i^{(s')} - 1$ and $l_{i}^{(s')} \neq i$ then we have $\lambda_{i-1,p'}^{(s')} = \lambda_{i,p'}^{(s')}$ except for $p' = l_i^{(s')}, i$.
The result follows from $\lambda_{i,l_i^{(s')}}^{(s')} = 1, \lambda_{i-1,l_i^{(s')}}^{(s')} = 0$ and $\lambda_{i,i}^{(s')} = 0$.
\end{itemize}
\end{proof}

\begin{lem}\label{lem:product}
Let $(x,y) \in R_{\lambda,i}$ and $\mu = \lambda \setminus (x,y)$.
Then we have
\begin{align*}
&\tilde{E}_{\lambda \mu} \tilde{F}_{\mu \lambda}\\
&= \displaystyle\prod_{(x',y') \in A_{\lambda,i}} \dfrac{\varepsilon_1 (x - x' + 1) + \varepsilon_2 (y - y' + 1)}{\varepsilon_1 (x - x') + \varepsilon_2 (y - y')} \displaystyle\prod_{(x',y') \in R_{\mu,i}} \dfrac{\varepsilon_1 (x - x' - 1) + \varepsilon_2 (y - y' - 1)}{\varepsilon_1 (x - x') + \varepsilon_2 (y - y')}.
\end{align*}
\end{lem}

\begin{proof}
The product $\tilde{E}_{\lambda \mu} \tilde{F}_{\mu \lambda}$ is written as
\[
 (-\hbar^{-2}) \times (\text{($s$)-factor}) \times \prod_{s' \neq s} (\text{($s'$)-factors})
.\]
We see that
\[
(-\hbar^{-2}) \times (\text{($s$)-factor}) = (-\hbar^{-2}) \times \left( \hbar^2(p-i-1) \dfrac{-1}{p-i-1} \right) = 1
\]
by Lemma~\ref{lem:factor}.
We classify ($s' \neq s$)-factors as follows;
\begin{enumerate}
\item $l_{i+1}^{(s')} = l_{i}^{(s')} = l_{i-1}^{(s')}$;
\item $l_{i+1}^{(s')} = l_{i}^{(s')}$, $l_{i-1}^{(s')} = l_{i}^{(s')} - 1$ and $l_{i}^{(s')} = i$;
\item $l_{i+1}^{(s')} = l_{i}^{(s')}$, $l_{i-1}^{(s')} = l_{i}^{(s')} - 1$ and $l_{i}^{(s')} \neq i$;
\item $l_{i+1}^{(s')} = l_{i}^{(s')} + 1$, $l_{i-1}^{(s')} = l_{i}^{(s')} - 1$ and $l_{i}^{(s')} = i$;
\item $l_{i+1}^{(s')} = l_{i}^{(s')} + 1$, $l_{i-1}^{(s')} = l_{i}^{(s')}$ and $l_{i}^{(s')} \neq i$;
\item $l_{i+1}^{(s')} = l_{i}^{(s')} + 1$, $l_{i-1}^{(s')} = l_{i}^{(s')} - 1$ and $l_{i}^{(s')} \neq i$.
\end{enumerate}
We calculate the factors in each case using Lemma~\ref{lem:factor}.
In (i), we have
\[
\Big( \hbar'(p - i - 1) - (a_s - a_{s'}) \Big) \times \dfrac{1}{\hbar'(p - i - 1) - (a_s - a_{s'})} = 1.
\]
In (ii), we have
\begin{align*}
& \Big( \hbar'(p - i - 1) - (a_s - a_{s'}) \Big) \times \dfrac{1}{\hbar'(p - i) - (a_s - a_{s'})} = \dfrac{\hbar'(p - l_{i}^{(s')} - 1) - (a_s - a_{s'})}{\hbar'(p - l_{i}^{(s')}) - (a_s - a_{s'})}.
\end{align*}
In (iii), we have
\begin{align*}
& \Big( \hbar'(p - i - 1) - (a_s - a_{s'}) \Big) \times \dfrac{\hbar'(p - l_{i}^{(s')} - 1) - (a_s - a_{s'})}{\Big( \hbar'(p - l_{i}^{(s')}) - (a_s - a_{s'}) \Big) \Big( \hbar'(p - i -1) - (a_s - a_{s'}) \Big)} \\
& = \dfrac{\hbar'(p - l_{i}^{(s')} - 1) - (a_s - a_{s'})}{\hbar'(p - l_{i}^{(s')}) - (a_s - a_{s'})}.
\end{align*}
In (iv), we have
\[
\Big( \hbar'(p - i) - (a_s - a_{s'}) \Big) \times \dfrac{1}{\hbar'(p - i) - (a_s - a_{s'})} = 1.
\]
In (v), we have
\begin{align*}
& \dfrac{\Big( \hbar'(p - l_{i}^{(s')}) - (a_s - a_{s'}) \Big) \Big( \hbar'(p - i -1) - (a_s - a_{s'}) \Big)}{\hbar'(p - l_{i}^{(s')} - 1) - (a_s - a_{s'})} \times \dfrac{1}{\hbar'(p - i - 1) - (a_s - a_{s'})} \\
&= \dfrac{\hbar'(p - l_{i}^{(s')}) - (a_s - a_{s'})}{\hbar'(p - l_{i}^{(s')} - 1) - (a_s - a_{s'})}.
\end{align*}
In (vi), we have
\begin{align*}
& \dfrac{\Big( \hbar'(p - l_{i}^{(s')}) - (a_s - a_{s'}) \Big) \Big( \hbar'(p - i -1) - (a_s - a_{s'}) \Big)}{\hbar'(p - l_{i}^{(s')} - 1) - (a_s - a_{s'})} \\
& \times \dfrac{\hbar'(p - l_{i}^{(s')} - 1) - (a_s - a_{s'})}{\Big( \hbar'(p - l_{i}^{(s')}) - (a_s - a_{s'}) \Big) \Big( \hbar'(p - i -1) - (a_s - a_{s'}) \Big)} \\
& = 1.
\end{align*}
Recall that $(s,p) \in R_{\Lambda,i}$ corresponds to $(x,y) \in R_{\lambda,i}$ with
\[
x = a - 1 = p_1 + \cdots + p_s - p,\ y = \lambda_{x+1} - 1.
\]
In (ii) and (iii), put
\[
x' = a' - 1 = p_1 + \cdots + p_{s'} - l_{i}^{(s')}.
\]
Then $(x', y' = \lambda_{x'+1}) \in A_{\lambda,i}$ and
\begin{align*}
\hbar'(p - l_{i}^{(s')}) - (a_s - a_{s'}) &= (\hbar'(p - 1) - a_{s}) - (\hbar'(l_{i}^ {(s')} - 1) - a_{s'}) \\
&= - (t m_a + \hbar'(a - \dfrac{1}{2})) + (t m_{a'} + \hbar'(a' - \dfrac{1}{2})) \\
&= \hbar(x - x') - t (m_a - m_{a'}) \\
&= \varepsilon_1 (x - x') + \varepsilon_2 (y - y').
\end{align*}
Hence we have
\begin{align*}
\prod_{s' \text{satisfies (ii) or (iii)}}\dfrac{\hbar'(p - l_{i}^{(s')} - 1) - (a_s - a_{s'})}{\hbar'(p - l_{i}^{(s')}) - (a_s - a_{s'})} = \prod_{(x',y') \in A_{\lambda,i}} \dfrac{ \varepsilon_1 (x - x' + 1) + \varepsilon_2 (y - y' + 1)}{ \varepsilon_1 (x - x') + \varepsilon_2 (y - y')}.
\end{align*}
In (v), put
\[
x' = a' - 1 = p_1 + \cdots + p_{s'} - (l_{i}^{(s')} + 1).
\]
Then $(x', y' = \lambda_{x'+1} - 1) \in R_{\mu,i}$ and
\[
\hbar'(p - l_{i}^{(s')}) - (a_s - a_{s'}) = \varepsilon_1 (x - x' - 1) + \varepsilon_2 (y - y' - 1).
\]
Hence we have
\[
\prod_{s' \text{satisfies (v)}}\dfrac{\hbar'(p - l_{i}^{(s')}) - (a_s - a_{s'})}{\hbar'(p - l_{i}^{(s')} - 1) - (a_s - a_{s'})} = \prod_{(x',y') \in R_{\mu,i}} \dfrac{ \varepsilon_1 (x - x' - 1) + \varepsilon_2 (y - y' - 1)}{ \varepsilon_1 (x - x') + \varepsilon_2 (y - y')}.
\]
This completes the proof.
\end{proof}

By Lemma~\ref{lem:ratio} (iii) and Lemma~\ref{lem:product} we have
\begin{align*}
& \left( E_{\lambda \mu}^{(r)} \right)^2 = \tilde{E}_{\lambda \mu}^{(r)} \tilde{F}_{\mu \lambda}^{(r)} \dfrac{\langle P_{\lambda}, P_{\lambda} \rangle_{F}}{\langle P_{\mu}, P_{\mu} \rangle_{F}} \\
&= \left( \varepsilon_1 (x + \dfrac{1}{2}i) + \varepsilon_2 (y - \dfrac{1}{2}i) \right)^{2r} \\
& \left( \displaystyle\prod_{(x',y') \in A_{\lambda,i, (x,y)}^{r}} \dfrac{\varepsilon_1 (x - x' + 1) + \varepsilon_2 (y - y' + 1)}{\varepsilon_1 (x - x') + \varepsilon_2 (y - y')} \displaystyle\prod_{(x',y') \in R_{\mu,i,(x,y)}^{r}} \dfrac{\varepsilon_1 (x - x' - 1) + \varepsilon_2 (y - y' - 1)}{\varepsilon_1 (x - x') + \varepsilon_2 (y - y')} \right)^2
\end{align*}
and hence
\begin{align*}
& E_{\lambda \mu}^{(r)} = c_{\lambda \mu}^{(r)} \left( \varepsilon_1 (x + \dfrac{1}{2}i) + \varepsilon_2 (y - \dfrac{1}{2}i) \right)^{r} \\
& \displaystyle\prod_{(x',y') \in A_{\lambda,i, (x,y)}^{r}} \dfrac{\varepsilon_1 (x - x' + 1) + \varepsilon_2 (y - y' + 1)}{\varepsilon_1 (x - x') + \varepsilon_2 (y - y')} \displaystyle\prod_{(x',y') \in R_{\mu,i,(x,y)}^{r}} \dfrac{\varepsilon_1 (x - x' - 1) + \varepsilon_2 (y - y' - 1)}{\varepsilon_1 (x - x') + \varepsilon_2 (y - y')}
\end{align*}
with some $c_{\lambda \mu}^{(r)} \in \{\pm 1\}$.
We have
\begin{align*}
& F_{\mu \lambda}^{(r)} = E_{\lambda \mu}^{(r)} \dfrac{\langle P_{\mu}, P_{\mu} \rangle_{F}}{\langle P_{\lambda}, P_{\lambda} \rangle_{F}} \\
&= c_{\lambda \mu}^{(r)} \left( \varepsilon_1 (x + \dfrac{1}{2}i) + \varepsilon_2 (y - \dfrac{1}{2}i) \right)^{r} \\
& \displaystyle\prod_{(x',y') \in A_{\lambda,i, (x,y)}^{l}} \dfrac{\varepsilon_1 (x - x' + 1) + \varepsilon_2 (y - y' + 1)}{\varepsilon_1 (x - x') + \varepsilon_2 (y - y')} \displaystyle\prod_{(x',y') \in R_{\mu,i,(x,y)}^{l}} \dfrac{\varepsilon_1 (x - x' - 1) + \varepsilon_2 (y - y' - 1)}{\varepsilon_1 (x - x') + \varepsilon_2 (y - y')}.
\end{align*}
Let us determine the signs $c_{\lambda \mu}^{(r)}$.
By the definition we have
\begin{align*}
E_{\lambda \mu}^{(r)} &= \dfrac{\alpha_{\lambda}}{\alpha_{\mu}} \tilde{E}_{\lambda \mu}^{(r)}\\
&= \dfrac{\alpha_{\lambda}}{\alpha_{\mu}} \left( \varepsilon_1 (x + \dfrac{1}{2}i) + \varepsilon_2 (y - \dfrac{1}{2}i) \right)^{r} \tilde{E}_{\lambda \mu}.
\end{align*}
This implies that $c_{\lambda \mu}^{(r)}$ is independent of $r$ since
\[
\varepsilon_1 (x + \dfrac{1}{2}i) + \varepsilon_2 (y - \dfrac{1}{2}i) \neq 0.
\]
Thus it is enough to determine $c_{\lambda \mu}^{(0)}$.
The value of $E_{\lambda \mu}^{(0)}$ has no pole at $\hbar = 0$, equivalently $\varepsilon_2 = - \varepsilon_1$, and specializes to $c_{\lambda \mu}^{(0)}$ at $\hbar = 0$.
When $\hbar = 0$, the Jack($\mathfrak{gl}_N$) symmetric function $P_{\lambda}$ specializes to the Schur symmetric function $s_{\lambda}$.
Then by the formula of $X_i^{+} s_{\lambda}$ we conclude $c_{\lambda \mu}^{(0)} = 1$.
We have obtained the following.

\begin{thm}\label{thm:main1}
The action of the Yangian $Y_{\hbar}(\mathfrak{sl}_N)$ on the Fock space $F$ is given by
\begin{align*}
& X_{i,r}^+ P_{\lambda} \\
&= \sum_{(x,y) \in R_{\lambda,i}} \left( \varepsilon_1 (x + \dfrac{1}{2}i) + \varepsilon_2 (y - \dfrac{1}{2}i) \right)^{r} \\
& \displaystyle\prod_{(x',y') \in A_{\lambda,i, (x,y)}^{r}} \dfrac{\varepsilon_1 (x - x' + 1) + \varepsilon_2 (y - y' + 1)}{\varepsilon_1 (x - x') + \varepsilon_2 (y - y')} \displaystyle\prod_{(x',y') \in R_{\lambda \setminus (x,y),i,(x,y)}^{r}} \dfrac{\varepsilon_1 (x - x' - 1) + \varepsilon_2 (y - y' - 1)}{\varepsilon_1 (x - x') + \varepsilon_2 (y - y')} P_{\lambda \setminus (x,y)},
\end{align*}
\begin{align*}
& X_{i,r}^- P_{\mu} \\
&= \sum_{(x,y) \in A_{\mu,i}} \left( \varepsilon_1 (x + \dfrac{1}{2}i) + \varepsilon_2 (y - \dfrac{1}{2}i) \right)^{r} \\
& \displaystyle\prod_{(x',y') \in A_{\mu \cup (x,y),i, (x,y)}^{l}} \dfrac{\varepsilon_1 (x - x' + 1) + \varepsilon_2 (y - y' + 1)}{\varepsilon_1 (x - x') + \varepsilon_2 (y - y')} \displaystyle\prod_{(x',y') \in R_{\mu,i,(x,y)}^{l}} \dfrac{\varepsilon_1 (x - x' - 1) + \varepsilon_2 (y - y' - 1)}{\varepsilon_1 (x - x') + \varepsilon_2 (y - y')} P_{\mu \cup (x,y)},
\end{align*}
\begin{align*}
& H_i(u) P_{\lambda}\\
&= \prod_{(x,y) \in A_{\lambda,i}}\dfrac{u - \varepsilon_1(x + \dfrac{1}{2}i - 1) - \varepsilon_2(y - \dfrac{1}{2}i - 1)}{u - \varepsilon_1(x + \dfrac{1}{2}i) - \varepsilon_2(y - \dfrac{1}{2}i)} \prod_{(x,y) \in R_{\lambda,i}}\dfrac{u - \varepsilon_1(x + \dfrac{1}{2}i + 1) - \varepsilon_2(y - \dfrac{1}{2}i + 1)}{u - \varepsilon_1(x + \dfrac{1}{2}i) - \varepsilon_2(y - \dfrac{1}{2}i)} P_{\lambda}.
\end{align*}
\end{thm}

\section{Quiver variety}\label{sec:quiver}

\subsection{Basics}
We consider the quiver varieties associated with the cyclic quiver $(I, \Omega)$, where the set of vertices is $I = \mathbb{Z} / N \mathbb{Z}$ and the set of arrows is $\Omega = \{\tau_i \colon i \to i+1 \mid i \in I\}$.
Let $\bar{\Omega} = \{\bar{\tau}_{i-1} \colon i \to i-1 \mid i \in I\}$ be the opposite arrows.

We recall the definition and basic properties of quiver varieties following \cite{MR1604167}.
For the special case of the cyclic quiver, see papers by Varagnolo-Vasserot~\cite{MR1722361} and Nagao~\cite{MR2583334,MR2517812}.
Take ${\bf v}=(v_i), {\bf w}=(w_i) \in (\mathbb{Z}_{\geq 0})^I$ and $I$-graded vector spaces $V$, $W$ such that $\dim V_i = v_i$, $\dim W_i = w_i$.
We set
\begin{align*}
M({\bf v}, {\bf w}) = &\bigoplus_{i \in I}\Hom(V_i, V_{i+1}) \oplus \bigoplus_{i \in I}\Hom(V_i, V_{i-1}) \\
& \oplus \bigoplus_{i \in I}\Hom(W_i, V_i) \oplus \bigoplus_{i \in I}\Hom(V_i, W_i)
\end{align*}
and write an element of $M({\bf v}, {\bf w})$ as $(B_{\Omega}=(B_i), B_{\bar{\Omega}}=(\bar{B}_i), a=(a_i), b=(b_i))$.
The moment map $\mu \colon M({\bf v}, {\bf w}) \to \bigoplus_{i \in I} \End(V_i)$ is defined by
\begin{align*}
\mu((B_{\Omega}, B_{\bar{\Omega}}, a, b)) = \sum_{i \in I} \left( B_{i-1}\bar{B}_{i} - \bar{B}_{i+1}B_{i} \right) + \sum_{i \in I} a_i b_i.
\end{align*}
An element $(B_{\Omega}, B_{\bar{\Omega}}, a, b) \in \mu^{-1}(0)$ is said to be stable if there does not exist any nonzero $(B_{\Omega}, B_{\bar{\Omega}})$-invariant subspace of $V$ contained in $\Ker b$.
Put $\mu^{-1}(0)^{s}$ to be the subset of $\mu^{-1}(0)$ consisting of stable elements.
Then $G_{\bf v} = \prod_{i \in I} GL(V_i)$ acts on $\mu^{-1}(0)^{s}$ freely.
The quiver variety is defined as
\[
\mathfrak{M}({\bf v}, {\bf w}) = \mu^{-1}(0)^{s} / G_{\bf v}
\]
and known to be a nonsingular quasi-projective variety.
In this paper, we only consider ${\bf w}$ corresponding to the basic weight $\Lambda_0$ of $\hat{\mathfrak{sl}}_N$, that is,
\[
w_i =
\begin{cases}
1 \text{ if } i=0,\\
0 \text{ otherwise.}
\end{cases}
\]
We denote $\mathfrak{M}({\bf v}, {\bf w})$ simply by $\mathfrak{M}({\bf v}) $ from now on.
Each ${\bf v}$ is identified with the element $\sum_{i \in I} v_i \alpha_i$ of the root lattice of $\hat{\mathfrak{sl}}_N$.

The quiver variety $\mathfrak{M}({\bf v})$ admits the action of the torus $T = (\mathbb{C}^{\times})^2$ defined by
\[
(t_1, t_2) [B_{\Omega}, B_{\bar{\Omega}}, a, b] = [t_1 B_{\Omega}, t_2 B_{\bar{\Omega}}, t_1 t_2 a, b]\]
for $(t_1, t_2) \in T$.
In the sequel, the natural representations of $T$ are denoted by $t_1, t_2$ which are defined as the first and the second projections.
We use the same symbols for the corresponding $T$-equivariant line bundles on $\mathfrak{M}({\bf v})$.

We define the tautological vector bundles $\mathcal{V}_i$, $\mathcal{W}_i$ by
\begin{align*}
\mathcal{V}_i &= \mu^{-1}(0)^{s} \times^{G_{\bf v}} V_i, \\
\mathcal{W}_i &= \mathcal{O}_{\mathfrak{M}({\bf v}, {\bf w})} \otimes W_i.
\end{align*}
We regard them as $T$-equivariant vector bundles.
A $T$-equivariant structure on $\mathcal{V}_i$ is naturally induced from the $T$-action on $\mathfrak{M}({\bf v})$ and $\mathcal{W}_i$ admits the trivial $T$-equivariant structure.
There exists a complex of $T$-equivariant vector bundles
\[
(t_1 t_2)^{-1} \mathcal{V}_i \to t_2^{-1} \mathcal{V}_{i+1} \oplus t_1^{-1} \mathcal{V}_{i-1} \oplus (t_1 t_2)^{-1} \mathcal{W}_i \to \mathcal{V}_i.
\]
We denote by $\mathcal{C}_i({\bf v})$ the corresponding class
\[
(t_1 t_2)^{-1} (t_1 \mathcal{V}_{i+1} + t_2 \mathcal{V}_{i-1} - (t_1 t_2 + 1)\mathcal{V}_i + \mathcal{W}_i)
\]
in the Grothendieck group.

For a fixed ${\bf v}$ and $i \in I$, put ${\bf v}^1 = {\bf v} - \alpha_i$ and ${\bf v}^2 = {\bf v}$.
We define a subset $\mathfrak{P}_i({\bf v})$ of $\mathfrak{M}({\bf v}^1) \times \mathfrak{M}({\bf v}^2)$ as follows.
An element of $\mathfrak{P}_i({\bf v})$ is a pair $([B_{\Omega}^1, B_{\bar{\Omega}}^1, a^1, b^1], [B_{\Omega}^2, B_{\bar{\Omega}}^2, a^2, b^2])$ such that
\begin{itemize}
\setlength{\itemsep}{-10pt}
\item $V^1$ is a $(B_{\Omega}^2, B_{\bar{\Omega}}^2)$-invariant subspace of $V^2$,\\
\item $\Ima a^2 \subseteq V^1$,\\
\item $(B_{\Omega}^2, B_{\bar{\Omega}}^2)|_{V^1} = (B_{\Omega}^1, B_{\bar{\Omega}}^1)$,\\
\item $b^2 |_{V^1} = b^1$.
\end{itemize}
Then $\mathfrak{P}_i({\bf v})$ is known to be nonsingular.
We define the line bundle $\mathcal{L}_i({\bf v})$ on $\mathfrak{P}_i({\bf v})$ to be the quotient $p_2^* \mathcal{V}_i^2 / p_1^* \mathcal{V}_i^1$, where $p_a \colon \mathfrak{M}({\bf v}^1) \times \mathfrak{M}({\bf v}^2) \to \mathfrak{M}({\bf v}^a)$ is the $a$-th projection for $a=1,2$.

For a $T$-equivariant vector bundle $E$, we denote by $c_i(E)$ the $i$-th $T$-equivariant Chern class and put $c_{-1/u}(E) = \sum_{i} c_i(E) (- 1/u)^i.$
The $T$-equivariant Euler class of $E$ is denoted by $e(E)$.

\subsection{Torus fixed points}\label{subsec:torus}

We describe the $T$-fixed points of $\mathfrak{M}({\bf v})$ and the tangent space at each fixed point.
These are well known and can be derived from the fact that $\mathfrak{M}({\bf v})$ is realized in the $\mathbb{Z} / N\mathbb{Z}$-fixed point subvariety of the Hilbert scheme of points on $\mathbb{C}^2$, for example.
For a partition $\lambda \in \mathcal{P}$, denote the number of $i$-cells in $\lambda$ by $v_i(\lambda)$ and put ${\bf v}(\lambda) = (v_i(\lambda)) \in (\mathbb{Z}_{\geq 0})^I$.
The $T$-fixed points of $\mathfrak{M}({\bf v})$ are parametrized by $\{ \lambda \in \mathcal{P} \mid {\bf v}(\lambda) = {\bf v} \}$.
Hence
\[
\bigsqcup_{\bf v} \mathfrak{M}({\bf v})^T \cong \mathcal{P}
\]
and we use the same symbols for fixed points and partitions.
The $T$-module structure of the tangent space $T_{\lambda} \mathfrak{M}({\bf v})$ at a fixed point $\lambda$ is given by
\[
T_{\lambda} \mathfrak{M}({\bf v}) = \bigoplus_{\substack{s \in \lambda\\ h_{\lambda}(s) \equiv 0 \bmod N}} \left( t_1^{l_{\lambda}(s) + 1} t_2^{-a_{\lambda}(s)} \oplus t_1^{-l_{\lambda}(s)} t_2^{a_{\lambda}(s) + 1} \right).
\]

Let $H_*^T (\mathfrak{M}({\bf v}))$ be the $T$-equivariant Borel-Moore homology group with $\mathbb{C}$-coefficient.
It is a module over $H_{T}^*(\text{pt}) = \mathbb{C}[\varepsilon_1, \varepsilon_2]$.
We consider the sum of the localized equivariant homology $H= \bigoplus_{{\bf v}} H_*^{T}(\mathfrak{M}({\bf v})) \otimes_{\mathbb{C}[\varepsilon_1, \varepsilon_2]} \mathbb{C}(\varepsilon_1, \varepsilon_2)$.
Let $i \colon \mathfrak{M}({\bf v})^T \to \mathfrak{M}({\bf v})$ be the inclusion.
Define $[\lambda] = {i}_{*}(1_{\lambda}) \in H_*^{T}(\mathfrak{M}({\bf v}))$ to be the class of the fixed point.
Then $\{ [\lambda] \}_{\lambda \in \mathcal{P}}$ forms a $\mathbb{C}(\varepsilon_1, \varepsilon_2)$-basis of $H$ by the localization theorem.
Define a symmetric bilinear form $\langle \ ,\ \rangle_{H}$ on $H$ by
\[
\langle \alpha, \beta \rangle_{H} = (-1)^{(1/2) \dim \mathfrak{M}({\bf v})}p_{*} ( i_*^{-1}\alpha \cap i_*^{-1}\beta)
\]
for $\alpha, \beta \in H_*^T(\mathfrak{M}({\bf v})) \otimes_{\mathbb{C}[\varepsilon_1, \varepsilon_2]} \mathbb{C}(\varepsilon_1, \varepsilon_2)$ and to be $0$ for different components, where $p \colon \mathfrak{M}({\bf v})^T \to \{\text{pt}\}$ is the projection.

\begin{prop}
Let $\lambda, \mu \in \mathcal{P}$.
Then we have
\[
\langle [\lambda], [\mu] \rangle_{H} = \delta_{\lambda \mu} \displaystyle \prod_{\substack{ s \in \lambda\\ h_{\lambda}(s) \equiv 0 \bmod N}} \Big( \varepsilon_1 (l_{\lambda}(s) + 1) - \varepsilon_2 a_{\lambda}(s) \Big) \Big( \varepsilon_1 l_{\lambda}(s) - \varepsilon_2 (a_{\lambda}(s) + 1) \Big).
\]
\end{prop}

\begin{proof}
We have 
\[
\langle [\lambda], [\mu] \rangle_{H} = \delta_{\lambda \mu} (-1)^{(1/2) \dim \mathfrak{M}({\bf v}(\lambda))} e(T_{\lambda} \mathfrak{M}({\bf v}(\lambda)))
\]
by a standard argument.
The assertion follows from the formula of the tangent space.
\end{proof}

This proposition together with the norm formula (\ref{eq:norm}) implies that the assignment
\[
\left( \displaystyle \prod_{\substack{s \in \lambda\\ h_{\lambda}(s) \equiv 0 \bmod N}} \dfrac{1}{\varepsilon_1 (l_{\lambda}(s) + 1) - \varepsilon_2 a_{\lambda}(s)} \right) [\lambda] \mapsto P_{\lambda}
\]
gives an isometry between the localized equivariant homology $H$ and the Fock space $F$.
However we note that giving an isometry via orthogonal bases has an ambiguity of sign.
We also remark that this correspondence can be considered without Yangian action since so is the definition of $P_{\lambda}$.
We should make an isometry $H \cong F$ more canonical by taking account of the Yangian action.

\section{Affine Yangian action on the Fock space}\label{sec:Fock}

\subsection{Affine Yangian}

\begin{dfn}
The affine Yangian $Y_{\varepsilon_1,\varepsilon_2}(\hat{\mathfrak{sl}}_N)$ is the algebra over $\mathbb{C}(\varepsilon_1,\varepsilon_2)$ generated by $x_{i,r}^{\pm}, h_{i,r}$ $(i \in \mathbb{Z} / N\mathbb{Z}, r \in \mathbb{Z}_{\geq 0})$ subject to the relations:
\begin{itemize}
\item if $N \geq 3$,
\begin{equation}
[h_{i,r}, h_{j,s}] = 0, \label{eq:1}
\end{equation}
\begin{equation}
[x_{i,r}^{+}, x_{j,s}^{-}] = \delta_{ij} h_{i, r+s}, \label{eq:2}
\end{equation}
\begin{equation}
[h_{i,0}, x_{j,r}^{\pm}] = \pm a_{ij} x_{j,r}^{\pm}, \label{eq:3}
\end{equation}
\begin{equation}
[h_{i, r+1}, x_{j, s}^{\pm}] - [h_{i, r}, x_{j, s+1}^{\pm}] = 0 \ (i \neq j, j \pm 1), \label{eq:4}
\end{equation}
\begin{equation}
[h_{i, r+1}, x_{i, s}^{\pm}] - [h_{i, r}, x_{i, s+1}^{\pm}] = \pm (\varepsilon_1 + \varepsilon_2) (h_{i, r} x_{i, s}^{\pm} + x_{i, s}^{\pm} h_{i, r}), \label{eq:5}
\end{equation}
\begin{equation}
[h_{i, r+1}, x_{i-1, s}^{+}] - [h_{i, r}, x_{i-1, s+1}^{+}] = - \varepsilon_1 h_{i, r} x_{i-1, s}^{+} - \varepsilon_2 x_{i-1, s}^{+} h_{i, r}, \label{eq:6}
\end{equation}
\begin{equation}
[h_{i, r+1}, x_{i+1, s}^{+}] - [h_{i, r}, x_{i+1, s+1}^{+}] = - \varepsilon_2 h_{i, r} x_{i+1, s}^{+} - \varepsilon_1 x_{i+1, s}^{+} h_{i, r}, \label{eq:7}
\end{equation}
\begin{equation}
[h_{i, r+1}, x_{i-1, s}^{-}] - [h_{i, r}, x_{i-1, s+1}^{-}] = \varepsilon_2 h_{i, r} x_{i-1, s}^{-} + \varepsilon_1 x_{i-1, s}^{-} h_{i, r}, \label{eq:8}
\end{equation}
\begin{equation}
[h_{i, r+1}, x_{i+1, s}^{-}] - [h_{i, r}, x_{i+1, s+1}^{-}] = \varepsilon_1 h_{i, r} x_{i+1, s}^{-} + \varepsilon_2 x_{i+1, s}^{-} h_{i, r}, \label{eq:9}
\end{equation}
\begin{equation}
[x_{i, r+1}^{\pm}, x_{j, s}^{\pm}] - [x_{i, r}^{\pm}, x_{j, s+1}^{\pm}] = 0 \ (i \neq j, j \pm 1), \label{eq:10}
\end{equation}
\begin{equation}
[x_{i, r+1}^{\pm}, x_{i, s}^{\pm}] - [x_{i, r}^{\pm}, x_{i, s+1}^{\pm}] = \pm (\varepsilon_1 + \varepsilon_2) (x_{i, r}^{\pm} x_{i, s}^{\pm} + x_{i, s}^{\pm} x_{i, r}^{\pm}), \label{eq:11}
\end{equation}
\begin{equation}
[x_{i, r+1}^{+}, x_{i-1, s}^{+}] - [x_{i, r}^{+}, x_{i-1, s+1}^{+}] = - \varepsilon_1 x_{i, r}^{+} x_{i-1, s}^{+} - \varepsilon_2 x_{i-1, s}^{+} x_{i, r}^{+}, \label{eq:12}
\end{equation}
\begin{equation}
[x_{i, r+1}^{+}, x_{i+1, s}^{+}] - [x_{i, r}^{+}, x_{i+1, s+1}^{+}] = - \varepsilon_2 x_{i, r}^{+} x_{i+1, s}^{+} - \varepsilon_1 x_{i+1, s}^{+} x_{i, r}^{+}, \label{eq:13}
\end{equation}
\begin{equation}
[x_{i, r+1}^{-}, x_{i-1, s}^{-}] - [x_{i, r}^{-}, x_{i-1, s+1}^{-}] = \varepsilon_2 x_{i, r}^{-} x_{i-1, s}^{-} + \varepsilon_1 x_{i-1, s}^{-} x_{i, r}^{-}, \label{eq:14}
\end{equation}
\begin{equation}
[x_{i, r+1}^{-}, x_{i+1, s}^{-}] - [x_{i, r}^{-}, x_{i+1, s+1}^{-}] = \varepsilon_1 x_{i, r}^{-} x_{i+1, s}^{-} + \varepsilon_2 x_{i+1, s}^{-} x_{i, r}^{-}, \label{eq:15}
\end{equation}
\begin{equation}
\sum_{w \in \mathfrak{S}_{1 - a_{ij}}}[x_{i,r_{w(1)}}^{\pm}, [x_{i,r_{w(2)}}^{\pm}, \dots, [x_{i,r_{w(1 - a_{ij})}}^{\pm}, x_{j,s}^{\pm}]\dots]] = 0 \ (i \neq j), \label{eq:16}
\end{equation}
where
\[
a_{ij} =
\begin{cases}
2  \text{ if } i=j, \\
-1 \text{ if } i=j \pm 1, \\
0  \text{ otherwise,}
\end{cases}
\]
\item if $N=2$,
\[
(\ref{eq:1}), (\ref{eq:2}), (\ref{eq:3}), (\ref{eq:5}), 
\]
\begin{align}
&[h_{i,r+2}, x_{i+1,s}^{\pm}] - 2[h_{i,r+1}, x_{i+1,s+1}^{\pm}] + [h_{i,r}, x_{i+1,s+2}^{\pm}] \pm (\varepsilon_1 + \varepsilon_2) ( h_{i,r+1} x_{i+1,s}^{\pm} + x_{i+1,s}^{\pm} h_{i,r+1} ) \nonumber \\
&\mp (\varepsilon_1 + \varepsilon_2) ( h_{i,r} x_{i+1,s+1}^{\pm} + x_{i+1,s+1}^{\pm} h_{i,r} ) + \varepsilon_1 \varepsilon_2 [h_{i,r}, x_{i+1,s}^{\pm}] = 0 , \label{eq:17}
\end{align}
\[
 (\ref{eq:11}),
\]
\begin{align}
&[x_{i,r+2}^{\pm}, x_{i+1,s}^{\pm}] - 2[x_{i,r+1}^{\pm}, x_{i+1,s+1}^{\pm}] + [x_{i,r}^{\pm}, x_{i+1,s+2}^{\pm}] \pm (\varepsilon_1 + \varepsilon_2) ( x_{i,r+1}^{\pm} x_{i+1,s}^{\pm} + x_{i+1,s}^{\pm} x_{i,r+1}^{\pm} ) \nonumber \\
&\mp (\varepsilon_1 + \varepsilon_2) ( x_{i,r}^{\pm} x_{i+1,s+1}^{\pm} + x_{i+1,s+1}^{\pm} x_{i,r}^{\pm} ) + \varepsilon_1 \varepsilon_2 [x_{i,r}^{\pm}, x_{i+1,s}^{\pm}] = 0, \label{eq:18}
\end{align}
\[
(\ref{eq:16}),
\]
where
\[
a_{ij} =
\begin{cases}
2  \text{ if } (i,j) = (0,0), (1,1), \\
-2 \text{ if } (i,j) = (0,1), (1,0). \\
\end{cases}
\]
\end{itemize}
\end{dfn}

We set
\begin{align*}
x_i^{\pm}(u) &= \hbar \sum_{r \geq 0} x_{i,r}^{\pm} u^{-r-1}, \\
h_i(u) &= 1 + \hbar \sum_{r \geq 0} h_{i,r} u^{-r-1}.
\end{align*}

\begin{rem}\label{rem:Guay}
The affine Yangian of type $A_{N-1}^{(1)}$ with $N \geq 3$ was introduced by Guay~\cite{MR2199856}.
The defining relations given in \cite{MR2199856} are slightly different from ours.
The algebra $\hat{Y}_{\lambda, \beta}$ in \cite[Definition~3.3 and Remark~3.4]{MR2199856} whose generators are denoted by $X_{i,r}^{\pm}, H_{i,r}$ $(i \in \mathbb{Z} / N\mathbb{Z}, r \in \mathbb{Z}_{\geq 0})$ is isomorphic to $Y_{\varepsilon_1,\varepsilon_2}(\hat{\mathfrak{sl}}_N)$ by
\begin{align*}
\hbar \sum_{r \geq 0} X_{i,r}^{\pm} u^{-r-1} &= x_i^{\pm}(u - \dfrac{1}{2}i(\varepsilon_1 - \varepsilon_2)), \\
1 + \hbar \sum_{r \geq 0} H_{i,r} u^{-r-1} &= h_i(u - \dfrac{1}{2}i(\varepsilon_1 - \varepsilon_2))
\end{align*}
for $i \neq 0$ and
\begin{align*}
\hbar \sum_{r \geq 0} X_{0,r}^{\pm} u^{-r-1} &= x_0^{\pm}(u - \dfrac{1}{4}N(\varepsilon_1 - \varepsilon_2)), \\
1 + \hbar \sum_{r \geq 0} H_{0,r} u^{-r-1} &= h_0(u - \dfrac{1}{4}N(\varepsilon_1 - \varepsilon_2))
\end{align*}
under $\lambda = \hbar$ and $\beta = \dfrac{1}{2}\hbar - \dfrac{1}{4}N(\varepsilon_1 - \varepsilon_2)$.
The relations for $X_{i,r}^{\pm}, H_{i,r}$ $(i \neq 0, r \in \mathbb{Z}_{\geq 0})$ in $\hat{Y}_{\lambda, \beta}$ are exactly the same as those in $Y_{\hbar}(\mathfrak{sl}_N)$.
This justifies that we use the same symbols $X_{i,r}^{\pm}, H_{i,r}$ both for elements of $Y_{\varepsilon_1,\varepsilon_2}(\hat{\mathfrak{sl}}_N) \cong \hat{Y}_{\lambda, \beta}$ and $Y_{\hbar}(\mathfrak{sl}_N)$.
Guay \cite[Corollary~6.1]{MR2323534} proved for $N \geq 4$ that the subalgebra of $\hat{Y}_{\lambda, \beta}$ generated by $X_{i,r}^{\pm}, H_{i,r}$ $(i \neq 0, r \in \mathbb{Z}_{\geq 0})$ is isomorphic to $Y_{\hbar}(\mathfrak{sl}_N)$.

When $N=2$, Boyarchenko-Levendorski{\u\i} \cite{MR1310291} defined the Yangian for $\hat{\mathfrak{sl}}_2$.
However the author does not know if our definition may or may not be equivalent to theirs when the parameters are specialized.
The relations above come from the consideration of the quiver variety so that Theorem~\ref{thm:Varagnolo} will hold.
As remarked in \cite[Remark~3.13]{MR2784748}, relations may be changed if we choose a different $T$-action.
If we define $X_{1,r}^{\pm}, H_{1,r}$ by
\begin{align*}
\hbar \sum_{r \geq 0} X_{1,r}^{\pm} u^{-r-1} &= x_1^{\pm}(u - \dfrac{1}{2}(\varepsilon_1 - \varepsilon_2)), \\
1 + \hbar \sum_{r \geq 0} H_{1,r} u^{-r-1} &= h_1(u - \dfrac{1}{2}(\varepsilon_1 - \varepsilon_2)),
\end{align*}
then they satisfy the defining relations of $Y_{\hbar}(\mathfrak{sl}_2)$.
\end{rem}

The elements $x_{i,0}^{\pm}, h_{i,0}$ $(i \in \mathbb{Z} / N \mathbb{Z})$ satisfy the defining relations of the affine Lie algebra $\hat{\mathfrak{sl}}_N$.
Hence any action of the affine Yangian $Y_{\varepsilon_1,\varepsilon_2}(\hat{\mathfrak{sl}}_N)$ restricts to actions of the Yangian $Y_{\hbar}(\mathfrak{sl}_N)$ and the affine Lie algebra $\hat{\mathfrak{sl}}_N$.

\subsection{Affine Yangian action on the homology of the quiver variety}

We recall an action of the affine Yangian $Y_{\varepsilon_1,\varepsilon_2}(\hat{\mathfrak{sl}}_N)$ on the localized equivariant homology $H$.
For each ${\bf v} \in (\mathbb{Z}_{\geq 0})^{I}$ and $i \in I$, let $p_a \colon \mathfrak{M}({\bf v}^1) \times \mathfrak{M}({\bf v}^2) \to \mathfrak{M}({\bf v}^a)$ denotes the $a$-th projection for $a=1,2$, where ${\bf v}^1 = {\bf v} - \alpha_i$ and ${\bf v}^2 = {\bf v}$.

\begin{thm}[Varagnolo~\cite{MR1818101}]\label{thm:Varagnolo}
The assignment
\begin{align*}
x_{i,r}^{+} &\mapsto \sum_{\bf v} (-1)^{\delta_{i,0} - (v_i - v_{i-1}) + 1} {p_1}_* (c_1(\mathcal{L}_i({\bf v}))^{r} \cap p_2^*(-)), \\
x_{i,r}^{-} &\mapsto \sum_{\bf v} (-1)^{v_i - v_{i+1}} {p_2}_* (c_1(\mathcal{L}_i({\bf v}))^{r} \cap p_1^*(-)), \\
h_i(u) &\mapsto \sum_{\bf v} \dfrac{c_{-1/u}(\mathcal{C}_i({\bf v}))}{c_{-1/u}(t_1 t_2 \mathcal{C}_i({\bf v}))} \cap (-)
\end{align*}
gives an action of the affine Yangian $Y_{\varepsilon_1,\varepsilon_2}(\hat{\mathfrak{sl}}_N)$ on the localized equivariant homology $H = \bigoplus_{{\bf v}} H_*^{T}(\mathfrak{M}({\bf v})) \otimes_{\mathbb{C}[\varepsilon_1, \varepsilon_2]} \mathbb{C}(\varepsilon_1, \varepsilon_2)$.
\end{thm}

\begin{rem}
In \cite{MR1818101}, the defining relations are checked only in the case $\varepsilon_1 = \varepsilon_2$ but a modification to the two-parameter case is straightforward.
We give a sketch of the proof for the case $N=2$ in the appendix.
\end{rem}

We can derive a formula for the actions of the generators of the affine Yangian $Y_{\varepsilon_1, \varepsilon_2}(\hat{\mathfrak{sl}}_N)$ on the fixed point basis in a way similar to the one used by Varagnolo-Vasserot~\cite{MR1722361} and Nagao~\cite{MR2583334} in the case of the equivariant $K$-group.
We give a proof for the completeness.

\begin{prop}\label{prop:formula}
The action of the affine Yangian $Y_{\varepsilon_1, \varepsilon_2}(\hat{\mathfrak{sl}}_N)$ on the localized equivariant homology $H$ is given by
\begin{align*}
& x_{i,r}^+ [\lambda] \\
&= \sum_{(x,y) \in R_{\lambda,i}}(-1)^{v_i(\lambda) - v_{i+1}(\lambda) + \# A_{\lambda,i} - \# R_{\lambda \setminus (x,y),i}} \left( \varepsilon_1 x + \varepsilon_2 y \right)^{r} \\
& \dfrac{\displaystyle\prod_{(x',y') \in A_{\lambda,i}} \varepsilon_1 (x - x' + 1) + \varepsilon_2 (y - y' + 1)}{\displaystyle\prod_{(x',y') \in R_{\lambda \setminus (x,y),i}} \varepsilon_1 (x - x') + \varepsilon_2 (y - y')} [\lambda \setminus (x,y)],
\end{align*}
\begin{align*}
& x_{i,r}^- [\mu] \\
&= \sum_{(x,y) \in A_{\mu,i}} (-1)^{v_i(\mu) - v_{i+1}(\mu) + 1 + \# A_{\mu \cup (x,y),i} - \# R_{\mu,i}} \left( \varepsilon_1 x + \varepsilon_2 y \right)^{r} \\
& \dfrac{\displaystyle\prod_{(x',y') \in R_{\mu,i}} \varepsilon_1 (x - x' - 1) + \varepsilon_2 (y - y' - 1)}{\displaystyle\prod_{(x',y') \in A_{\mu \cup (x,y),i}} \varepsilon_1 (x - x') + \varepsilon_2 (y - y')} [\mu \cup (x,y)],
\end{align*}
\begin{align*}
& h_i(u) [\lambda]\\
&= \prod_{(x,y) \in A_{\lambda,i}}\dfrac{u - (\varepsilon_1(x - 1) + \varepsilon_2(y - 1))}{u - (\varepsilon_1 x + \varepsilon_2 y)} \prod_{(x,y) \in R_{\lambda,i}}\dfrac{u - (\varepsilon_1 (x + 1) + \varepsilon_2 (y + 1))}{u - (\varepsilon_1 x + \varepsilon_2  y)} [\lambda].
\end{align*}
\end{prop}

\begin{proof}
Let $N_{\mu \lambda}$ be the fiber at $(\mu, \lambda)$ of the normal bundle to $\mathfrak{P}_i({\bf v}) \subseteq \mathfrak{M}({\bf v}^1) \times \mathfrak{M}({\bf v}^2)$.
We put $T_{\lambda} = T_{\lambda}\mathfrak{M}({\bf v})$.
The coefficient of $[\mu]$ in $x_{i,r}^+[\lambda]$ is given by
\begin{align*}
\dfrac{\langle x_{i,r}^+[\lambda], [\mu] \rangle_{H}}{\langle [\mu], [\mu] \rangle_{H}} =
\delta((\mu, \lambda) \in \mathfrak{P}_i({\bf v}))(-1)^{\delta_{i,0} - (v_i(\lambda) - v_{i-1}(\lambda)) + 1} c_1(\mathcal{L}_i({\bf v}))^r \dfrac{e(N_{\mu \lambda})}{e(T_{\mu})}.
\end{align*}
Similarly the coefficient of $[\lambda]$ in $x_{i,r}^-[\mu]$ is given by
\begin{align*}
\dfrac{\langle x_{i,r}^-[\mu], [\lambda] \rangle_{H}}{\langle [\lambda], [\lambda] \rangle_{H}} =
\delta((\mu, \lambda) \in \mathfrak{P}_i({\bf v})) (-1)^{v_i(\lambda) - v_{i+1}(\lambda)} c_1(\mathcal{L}_i({\bf v}))^r \dfrac{e(N_{\mu \lambda})}{e(T_{\lambda})}.
\end{align*}
Note that $(\mu,\lambda) \in \mathfrak{P}_i({\bf v})$ if and only if $\mu$ is obtained from $\lambda$ by removing an $i$-cell with ${\bf v}^1 = {\bf v}(\mu)$ and ${\bf v}^2 = {\bf v}(\lambda)$.

For a $T$-module $E = \bigoplus_{a,b} t_1^a t_2^b$, we define $E_i$ by
\[
E_i = \bigoplus_{b-a \equiv i \bmod N} t_1^a t_2^b.
\]
For $\lambda \in \mathcal{P}$, put
\[
V_{\lambda} = \bigoplus_{(x,y) \in \lambda} t_1^x t_2^y.
\]
Then the fiber at $\lambda$ of the tautological bundle $\mathcal{V}_i$ is isomorphic to $(V_{\lambda})_i$ as a $T$-module.
We regard $W$ as the trivial one-dimensional $T$-module.
In the sequel, equalities of $T$-modules are considered in the Grothendieck group.

Suppose $(x,y) \in R_{\lambda,i}$ and $\mu = \lambda \setminus (x,y)$.
Then we have
\[
V_{\lambda} - V_{\mu} = t_1^x t_2^y
\]
and
\[
c_1(\mathcal{L}_i({\bf v})) = \varepsilon_1 x + \varepsilon_2 y.
\]
We use equalities proved by Varagnolo-Vasserot \cite[Lemma 7 and 8]{MR1722361}:
\[
\bigl( (t_1 + t_2 - t_1 t_2 - 1)V_{\lambda} + W \bigr)_i = \sum_{(x',y') \in A_{\lambda,i}} t_1^{x'} t_2^{y'} - \sum_{(x',y') \in R_{\lambda,i}} t_1^{x'+1} t_2^{y'+1},
\]
\[
\bigl( (t_1 + t_2 - t_1 t_2 - 1)V_{\lambda}^* + t_1 t_2 W^* \bigr)_{-i} = \sum_{(x',y') \in A_{\lambda,i}} t_1^{-x'+1} t_2^{-y'+1} - \sum_{(x',y') \in R_{\lambda,i}} t_1^{-x'} t_2^{-y'},
\]
\[
N_{\mu \lambda} = \left( (t_1 + t_2 - t_1 t_2 - 1) V_{\mu}^* V_{\lambda} +  t_1 t_2 W^* V_{\lambda} + V_{\mu}^* W - t_1 t_2 \right)_0,
\]
\[
T_{\mu} = \bigl( (t_1 + t_2 - t_1 t_2 - 1) V_{\mu}^* V_{\mu} +  t_1 t_2 W^* V_{\mu} + V_{\mu}^* W \bigr)_0,
\]
\[
T_{\lambda} = \bigl( (t_1 + t_2 - t_1 t_2 - 1) V_{\lambda}^* V_{\lambda} +  t_1 t_2 W^* V_{\lambda} + V_{\lambda}^* W \bigr)_0.
\]
Using above formulas, we calculate $N_{\mu \lambda} - T_{\mu}$ and $N_{\mu \lambda} - T_{\lambda}$ as follows:
\begin{align*}
N_{\mu \lambda} - T_{\mu} &= \left( (t_1 + t_2 - t_1 t_2 - 1) V_{\mu}^* (V_{\lambda} - V_{\mu}) + t_1 t_2 W^* (V_{\lambda} - V_{\mu}) - t_1 t_2 \right)_0 \\
&= t_1^x t_2^y \left( (t_1 + t_2 - t_1 t_2 - 1) V_{\mu}^* + t_1 t_2 W^* \right)_{-i} - t_1 t_2 \\
&= t_1^x t_2^y \left( \sum_{(x',y') \in A_{\mu,i}} t_1^{-x'+1} t_2^{-y'+1} - \sum_{(x',y') \in R_{\mu,i}} t_1^{-x'} t_2^{-y'} \right) - t_ 1 t_2 \\
&= \sum_{(x',y') \in A_{\lambda,i}} t_1^{x-x'+1} t_2^{y-y'+1} - \sum_{(x',y') \in R_{\mu,i}} t_1^{x-x'} t_2^{y-y'},
\end{align*}
\begin{align*}
N_{\mu \lambda} - T_{\lambda} &= \left( (t_1 + t_2 - t_1 t_2 - 1) V_{\lambda} (V_{\mu}^* - V_{\lambda}^*) + W (V_{\mu}^* - V_{\lambda}^*) - t_1 t_2 \right)_0 \\
&= - t_1^{-x} t_2^{-y} \bigl( (t_1 + t_2 - t_1 t_2 - 1)V_{\lambda} + W \bigr)_i - t_1 t_2 \\
&= - t_1^{-x} t_2^{-y} \left( \sum_{(x',y') \in A_{\lambda,i}} t_1^{x'} t_2^{y'} - \sum_{(x',y') \in R_{\lambda,i}} t_1^{x'+1} t_2^{y'+1} \right) - t_ 1 t_2 \\
&= - \sum_{(x',y') \in A_{\lambda,i}} t_1^{-x+x'} t_2^{-y+y'} + \sum_{(x',y') \in R_{\mu,i}} t_1^{-x+x'+1} t_2^{-y+y'+1}.
\end{align*}
Hence we have
\[
\dfrac{e(N_{\mu \lambda})}{e(T_{\mu})} = \dfrac{\displaystyle\prod_{(x',y') \in A_{\lambda,i}} \varepsilon_1 (x - x' + 1) + \varepsilon_2 (y - y' + 1)}{\displaystyle\prod_{(x',y') \in R_{\mu,i}} \varepsilon_1 (x - x') + \varepsilon_2 (y - y')}
\]
and
\[
\dfrac{e(N_{\mu \lambda})}{e(T_{\lambda})} = (-1)^{\# A_{\lambda,i} - \# R_{\mu,i}} \dfrac{\displaystyle\prod_{(x',y') \in R_{\mu,i}} \varepsilon_1 (x - x' - 1) + \varepsilon_2 (y - y' - 1)}{\displaystyle\prod_{(x',y') \in A_{\lambda,i}} \varepsilon_1 (x - x') + \varepsilon_2 (y - y')},
\]
which complete the proof for the actions of $x_{i,r}^{\pm}$ with the following easy observation:
\[
2v_i(\lambda) - v_{i+1}(\lambda) - v_{i-1}(\lambda) = \delta_{i,0} - \# A_{\lambda,i} + \# R_{\lambda,i}.
\]

For the action of $h_{i}(u)$, we see that
\begin{align*}
\mathcal{C}_i({\bf v}) |_{\lambda} &= (t_1 t_2)^{-1} \bigl( (t_1 + t_2 - t_1 t_2 - 1)V_{\lambda} + W \bigr)_i \\
&= \sum_{(x,y) \in A_{\lambda,i}} t_1^{x-1} t_2^{y-1} - \sum_{(x,y) \in R_{\lambda,i}} t_1^{x} t_2^{y},
\end{align*}
which completes the proof.
\end{proof}

Put
\[
b_{\lambda}' = \left( \displaystyle \prod_{\substack{s \in \lambda\\ h_{\lambda}(s) \equiv 0 \bmod N}} \dfrac{1}{\varepsilon_1 (l_{\lambda}(s) + 1) - \varepsilon_2 a_{\lambda}(s)} \right) [\lambda].
\]
Then by Lemma~\ref{lem:ratio} (ii) and Proposition~\ref{prop:formula} we obtain the following formula:
\begin{align*}
& x_{i,r}^+ b_{\lambda}' \\
&= \sum_{(x,y) \in R_{\lambda,i}} (-1)^{v_{i}(\lambda) - v_{i+1}(\lambda) + \# A_{\lambda,i, (x,y)}^{l} - \# R_{\lambda \setminus (x,y),i, (x,y)}^{l}} \left( \varepsilon_1 x + \varepsilon_2 y \right)^{r} \\
& \displaystyle\prod_{(x',y') \in A_{\lambda,i, (x,y)}^{r}} \dfrac{\varepsilon_1 (x - x' + 1) + \varepsilon_2 (y - y' + 1)}{\varepsilon_1 (x - x') + \varepsilon_2 (y - y')} \displaystyle\prod_{(x',y') \in R_{\lambda \setminus (x,y),i,(x,y)}^{r}} \dfrac{\varepsilon_1 (x - x' - 1) + \varepsilon_2 (y - y' - 1)}{\varepsilon_1 (x - x') + \varepsilon_2 (y - y')} b_{\lambda \setminus (x,y)}',
\end{align*}
\begin{align*}
& x_{i,r}^- b_{\mu}' \\
&= \sum_{(x,y) \in A_{\mu,i}} (-1)^{v_{i}(\mu) - v_{i+1}(\mu) + 1 + \# A_{\mu \cup (x,y),i, (x,y)}^{l} - \# R_{\mu,i, (x,y)}^{l}} \left( \varepsilon_1 x + \varepsilon_2 y \right)^{r} \\
& \displaystyle\prod_{(x',y') \in A_{\mu \cup (x,y),i, (x,y)}^{l}} \dfrac{\varepsilon_1 (x - x' + 1) + \varepsilon_2 (y - y' + 1)}{\varepsilon_1 (x - x') + \varepsilon_2 (y - y')} \displaystyle\prod_{(x',y') \in R_{\mu,i,(x,y)}^{l}} \dfrac{\varepsilon_1 (x - x' - 1) + \varepsilon_2 (y - y' - 1)}{\varepsilon_1 (x - x') + \varepsilon_2 (y - y')} b_{\mu \cup (x,y)}'.
\end{align*}

We shall renormalize $b_{\lambda}'$ so that the signs in the above formulas are all $1$.
For each $\lambda \in \mathcal{P}$, let $d$ be the minimum odd integer such that $l(\lambda) < Nd$ and put
\[
\varepsilon_{\lambda} = \# \{ (a,b) \mid 1 \leq a < b \leq Nd,\ j(\lambda)_a  \geq j(\lambda)_b \}.
\]

\begin{lem}\label{lem:sign}
Let $\lambda, \mu \in \mathcal{P}$ such that $\mu$ is obtained from $\lambda$ by removing an $i$-cell.
Then we have
\[
\varepsilon_{\lambda} - \varepsilon_{\mu} \equiv v_i(\lambda) - v_{i+1}(\lambda) + \# A_{\lambda,i}^{l} - \# R_{\mu,i}^{l} \bmod 2.
\]
\end{lem}

\begin{proof}
By the assumption there exists a unique $a$ such that $j(\lambda)_a = i+1$, $j(\mu)_a = i$ and $j(\lambda)_b = j(\mu)_b$ if $b \neq a$.
Put $j_b = j(\lambda)_b$ for every $b$.
Then we easily see that
\[
\varepsilon_{\lambda} - \varepsilon_{\mu} = - \# \{ b \mid 1 \leq b < a,\ j_{b} = i\} + \# \{ b \mid a < b \leq Nd,\ j_{b} = i+1\}.
\]
We have
\begin{align*}
v_i(\lambda) - v_{i+1}(\lambda) &= \# \{ b \mid 1 \leq b \leq l(\lambda),\ j_{b} = i+1 \} - \# \{ b \mid 1 \leq b \leq l(\lambda),\ - (b-1) \equiv i+1 \bmod N \} \\
&= \# \{ b \mid 1 \leq b \leq Nd,\ j_{b} = i+1 \} - \# \{ b \mid 1 \leq b \leq Nd,\ - (b-1) \equiv i+1 \bmod N \} \\
&= \# \{ b \mid 1 \leq b \leq Nd,\ j_{b} = i+1 \} - d
\end{align*}
and
\begin{align*}
\# A_{\lambda,i}^{l} - \# R_{\mu,i}^{l} = \# \{ b \mid 1 \leq b < a,\ j_{b} = i \} - \# \{ b \mid 1 \leq b < a,\ j_{b} = i+1 \},
\end{align*}
hence
\begin{align*}
& v_i(\lambda) - v_{i+1}(\lambda) + \# A_{\lambda,i}^{l} - \# R_{\mu,i}^{l} \\
&= \# \{ b \mid 1 \leq b < a,\ j_{b} = i \} + \# \{ b \mid a \leq b \leq Nd,\ j_{b} = i+1 \} - d \\
& \equiv \# \{ b \mid 1 \leq b < a,\ j_{b} = i \} + \# \{ b \mid a < b \leq Nd,\ j_{b} = i+1 \}
\end{align*}
since $d$ is taken to be odd.
\end{proof}

Define
\[
b_{\lambda} = (-1)^{\varepsilon_{\lambda}} b_{\lambda}'.
\]
Then we obtain the following formula by Lemma~\ref{lem:sign}.

\begin{thm}\label{thm:formula}
The action of the affine Yangian $Y_{\varepsilon_1, \varepsilon_2}(\hat{\mathfrak{sl}}_N)$ on the localized equivariant homology $H$ is given by
\begin{align*}
& x_{i,r}^+ b_{\lambda} \\
&= \sum_{(x,y) \in R_{\lambda,i}} \left( \varepsilon_1 x + \varepsilon_2 y \right)^{r} \\
& \displaystyle\prod_{(x',y') \in A_{\lambda,i, (x,y)}^{r}} \dfrac{\varepsilon_1 (x - x' + 1) + \varepsilon_2 (y - y' + 1)}{\varepsilon_1 (x - x') + \varepsilon_2 (y - y')} \displaystyle\prod_{(x',y') \in R_{\lambda \setminus (x,y),i,(x,y)}^{r}} \dfrac{\varepsilon_1 (x - x' - 1) + \varepsilon_2 (y - y' - 1)}{\varepsilon_1 (x - x') + \varepsilon_2 (y - y')} b_{\lambda \setminus (x,y)},
\end{align*}
\begin{align*}
& x_{i,r}^- b_{\mu} \\
&= \sum_{(x,y) \in A_{\mu,i}} \left( \varepsilon_1 x + \varepsilon_2 y \right)^{r} \\
& \displaystyle\prod_{(x',y') \in A_{\mu \cup (x,y),i, (x,y)}^{l}} \dfrac{\varepsilon_1 (x - x' + 1) + \varepsilon_2 (y - y' + 1)}{\varepsilon_1 (x - x') + \varepsilon_2 (y - y')} \displaystyle\prod_{(x',y') \in R_{\mu,i,(x,y)}^{l}} \dfrac{\varepsilon_1 (x - x' - 1) + \varepsilon_2 (y - y' - 1)}{\varepsilon_1 (x - x') + \varepsilon_2 (y - y')} b_{\mu \cup (x,y)},
\end{align*}
\begin{align*}
& h_i(u) b_{\lambda}\\
&= \prod_{(x,y) \in A_{\lambda,i}}\dfrac{u - (\varepsilon_1(x - 1) + \varepsilon_2(y - 1))}{u - (\varepsilon_1 x + \varepsilon_2 y)} \prod_{(x,y) \in R_{\lambda,i}}\dfrac{u - (\varepsilon_1 (x + 1) + \varepsilon_2 (y + 1))}{u - (\varepsilon_1 x + \varepsilon_2  y)} b_{\lambda}.
\end{align*}
\end{thm}

\subsection{Isomorphism}

The following is the main result of this paper.

\begin{thm}\label{thm:main2}
The assignment
\[
b_{\lambda} \mapsto P_{\lambda}
\]
gives a $\mathbb{C}(\varepsilon_1, \varepsilon_2)$-linear isometry and an isomorphism of modules over the Yangian $Y_{\hbar}(\mathfrak{sl}_N)$ and the affine Lie algebra $\hat{\mathfrak{sl}}_N$ between the localized equivariant homology $H$ and the Fock space $F$.
\end{thm}

\begin{proof}
We have already mentioned the isometry in Subsection~\ref{subsec:torus}.

Recall Remark~\ref{rem:Guay} that the inclusion $Y_{\hbar}(\mathfrak{sl}_{N}) \subset Y_{\varepsilon_1, \varepsilon_2}(\hat{\mathfrak{sl}}_N)$ is given by
\begin{align*}
\hbar \sum_{r \geq 0} X_{i,r}^{\pm} u^{-r-1} &= x_i^{\pm}(u - \dfrac{1}{2}i(\varepsilon_1 - \varepsilon_2)), \\
1 + \hbar \sum_{r \geq 0} H_{i,r} u^{-r-1} &= h_i(u - \dfrac{1}{2}i(\varepsilon_1 - \varepsilon_2)).
\end{align*}
Hence Theorem~\ref{thm:main1} and \ref{thm:formula} imply that the above assignment gives a $Y_{\hbar}(\mathfrak{sl}_N)$-module isomorphism.

To see that this assignment gives an $\hat{\mathfrak{sl}}_N$-module isomorphism, we check that the actions of $x_{i,0}^{\pm}, h_{i,0}$ on the bases of both sides coincide.
We can define the affine Yangian $Y_{\varepsilon_1, \varepsilon_2}(\hat{\mathfrak{sl}}_N)$ over $\mathbb{C}[\varepsilon_1, \varepsilon_2]$ and its subalgebra $\langle x_{i,0}^{\pm}, h_{i,0} \mid i \in \mathbb{Z} / N \mathbb{Z} \rangle$ does not change after the specialization $\varepsilon_1 = - \varepsilon_2$.
Hence a formula for the actions of $x_{i,0}^{\pm}, h_{i,0}$ on the Schur symmetric functions $s_{\lambda}$ is obtained from the formula in Theorem~\ref{thm:formula} by putting $\varepsilon_1 = - \varepsilon_2$.
The result coincides with the original action of $\hat{\mathfrak{sl}}_N$ on the Fock space, as desired.
\end{proof}

\begin{cor}\label{cor:cor}
The actions of the Yangian $Y_{\hbar}(\mathfrak{sl}_N)$ and the affine Lie algebra $\hat{\mathfrak{sl}}_N$ on the Fock space $F$ can be uniquely extended to an action of the affine Yangian $Y_{\varepsilon_1, \varepsilon_2}(\hat{\mathfrak{sl}}_N)$.
\end{cor}

\begin{proof}
The actions of $Y_{\hbar}(\mathfrak{sl}_N)$ and $\hat{\mathfrak{sl}}_N$ can be extended to that of $Y_{\varepsilon_1, \varepsilon_2}(\hat{\mathfrak{sl}}_N)$ by Theorem~\ref{thm:main2}.
Uniqueness follows from the fact that $Y_{\varepsilon_1, \varepsilon_2}(\hat{\mathfrak{sl}}_N)$ is generated by its subalgebras 
$\langle x_{i,r}^{\pm}, h_{i,r} \mid i \neq 0, r \in \mathbb{Z}_{\geq 0} \rangle$ and 
$\langle x_{i,0}^{\pm}, h_{i,0} \mid i \in \mathbb{Z} / N \mathbb{Z} \rangle$.
\end{proof}

\appendix
\section*{Appendix}

This appendix is devoted to give a sketch of the proof of Theorem~\ref{thm:Varagnolo} for the case $N=2$ following arguments in \cite[Section~5]{MR1818101}.
In \cite[Section~4 Remark]{MR1818101}, relations for the Yangian associated with a general symmetric Kac-Moody Lie algebra are given.
Since a $T$-action used there is different from ours, those relations for $\hat{\mathfrak{sl}}_2$ do not coincide with ours even if the parameters are specialized to $\varepsilon_1 = \varepsilon_2$.

\begin{proof}
Modifications are needed only for the relations (\ref{eq:17}) and (\ref{eq:18}).
Thus we put $k=i$, $l=i+1$ in \cite[Section~5]{MR1818101}.
We prove only the $+$-cases.

For (\ref{eq:17}), we take ${\bf v}^1, {\bf v}^2$ such that ${\bf v}^2 = {\bf v}^1 + \alpha_{i+1}$.
Set
\[
c_{i+1} = c_1(\mathcal{L}_{i+1}({\bf v}^2))
\]
and
\begin{align*}
X &= c_{-1/u}\left( ( \mathcal{C}_i({\bf v}^1) - t_1 t_2 \mathcal{C}_i({\bf v}^1)) - ( \mathcal{C}_i({\bf v}^2) - t_1 t_2 \mathcal{C}_i({\bf v}^2) ) \right) \\
&= c_{-1/u}\left( (t_1 + t_2 - t_1^{-1} - t_2^{-1}) \mathcal{L}_{i+1}({\bf v}^2) \right) \\
&= \dfrac{(1 - u^{-1}( c_{i+1} + \varepsilon_1 ))(1 - u^{-1}( c_{i+1} + \varepsilon_2 ))}{(1 - u^{-1}( c_{i+1} - \varepsilon_1 ))(1 - u^{-1}( c_{i+1} - \varepsilon_2 ))}.
\end{align*}
Then, by an argument similar to the one given in \cite[p.\ 278 Relation (1.2)]{MR1818101}, a proof of the relation~(\ref{eq:17}) is reduced to the identity
\[
(X - 1)(u - c_{i+1})^2 + (\varepsilon_1 + \varepsilon_2)(X + 1)(u - c_{i+1}) + \varepsilon_1 \varepsilon_2 (X - 1) = 0.
\]
This identity can be checked in a straightforward way.

For (\ref{eq:18}), we take ${\bf v}^1, {\bf v}^3$ such that ${\bf v}^3 = {\bf v}^1 + \alpha_{i} + \alpha_{i+1}$.
We have the corresponding vector bundles $\mathcal{V}_{i}^1$, $\mathcal{V}_{i+1}^1$, $\mathcal{V}_{i}^3$, and $\mathcal{V}_{i+1}^3$.
Set
\begin{align*}
\mathcal{E}_{i,i+1} &= (t_1 + t_2) \Hom(\mathcal{V}_{i+1}^3 / \mathcal{V}_{i+1}^1, \mathcal{V}_{i}^3 / \mathcal{V}_{i}^1), \\
\mathcal{E}_{i+1,i} &= (t_1 + t_2) \Hom(\mathcal{V}_{i}^3 / \mathcal{V}_{i}^1, \mathcal{V}_{i+1}^3 / \mathcal{V}_{i+1}^1).
\end{align*}
By an argument similar to the one given in \cite[p.\ 279 Relation (1.4) with $k \neq l$]{MR1818101}, we have
\[
c_1(\mathcal{E}_{i,i+1}) x_{i,r}^+ x_{i+1,s}^+ = c_1(\mathcal{E}_{i+1,i}) x_{i+1,s}^+ x_{i,r}^+.
\]
Put
\[
c_{i} = c_1(\mathcal{V}_i^3 / \mathcal{V}_i^1), \ c_{i+1} = c_1(\mathcal{V}_{i+1}^3 / \mathcal{V}_{i+1}^1).
\]
Then we obtain the identity
\begin{align*}
( c_{i} - c_{i+1} + \varepsilon_1 )( c_{i} - c_{i+1} + \varepsilon_2 ) x_{i,r}^+ x_{i+1,s}^+ = ( c_{i+1} - c_{i} + \varepsilon_1 )( c_{i+1} - c_{i} + \varepsilon_2 ) x_{i+1,s}^+ x_{i,r}^+.
\end{align*}
This is exactly the relation (\ref{eq:18}).
\end{proof}

\def\cprime{$'$}
\providecommand{\bysame}{\leavevmode\hbox to3em{\hrulefill}\thinspace}
\providecommand{\MR}{\relax\ifhmode\unskip\space\fi MR }
\providecommand{\MRhref}[2]{%
  \href{http://www.ams.org/mathscinet-getitem?mr=#1}{#2}
}
\providecommand{\href}[2]{#2}


\end{document}